\documentclass[12pt,reqno]{amsart}
\usepackage[a4paper,
  left=25mm,
  right=25mm,
  top=25mm,
  bottom=25mm
]{geometry}

\usepackage{amsmath,amsfonts,amssymb, amsthm}
\usepackage{bm}
\usepackage{graphicx}
\usepackage{ascmac}
\usepackage[legacycolonsymbols]{mathtools}
\usepackage{nccmath}
\usepackage{mathrsfs}
\usepackage{mleftright}

\usepackage[italicdiff]{physics}

\usepackage{thmtools}
\usepackage{thm-restate}
\usepackage{needspace}
\usepackage{etoolbox}
\usepackage{xspace}
\usepackage{placeins}
\usepackage{hyperref}
\usepackage{multirow}
\usepackage{makecell}

\usepackage{enumitem}
\setlist[enumerate]{label=(\arabic*), ref=(\arabic*), leftmargin=2.63em, labelsep=0.35em}
\setlist[itemize]{leftmargin=2em, labelsep=0.4em}
\mleftright
\allowdisplaybreaks
\mathtoolsset{showonlyrefs=true}
\def\thefootnote{\arabic{footnote}}

\DeclareMathAlphabet{\mathpzc}{OT1}{pzc}{m}{it}

\makeatletter

\newif\if@noindentafterheading
\let\orig@afterheading\@afterheading
\def\@afterheading{%
  \if@noindentafterheading\global\@afterindentfalse\fi
  \orig@afterheading
}

\renewcommand\section{\@startsection{section}{1}{\z@}%
  {-3.5ex \@plus -1ex \@minus -.2ex}
  {2.6ex \@plus .2ex}
  {\normalfont\large\bfseries}%
}

\renewcommand\subsection{\@startsection{subsection}{2}{\z@}%
  {-2.5ex \@plus -.3ex \@minus -.2ex}%
  {1.5ex \@plus .2ex}%
  {\normalfont\bfseries}%
}

\let\orig@section\section
\renewcommand\section{%
  \@ifstar{\section@star}{\section@nostar}%
}
\newcommand\section@nostar{%
  \@ifnextchar[{\section@opt}{\section@noopt}%
}
\newcommand\section@opt[2][]{%
  \global\@noindentafterheadingtrue
  \orig@section[#1]{#2}%
  \global\@noindentafterheadingfalse
}
\newcommand\section@noopt[1]{%
  \global\@noindentafterheadingtrue
  \orig@section{#1}%
  \global\@noindentafterheadingfalse
}
\newcommand\section@star[1]{%
  \global\@noindentafterheadingtrue
  \orig@section*{#1}%
  \global\@noindentafterheadingfalse
}

\let\orig@subsection\subsection
\renewcommand\subsection{%
  \@ifstar{\subsection@star}{\subsection@nostar}%
}
\newcommand\subsection@nostar{%
  \@ifnextchar[{\subsection@opt}{\subsection@noopt}%
}
\newcommand\subsection@opt[2][]{%
  \global\@noindentafterheadingtrue
  \orig@subsection[#1]{#2}%
  \global\@noindentafterheadingfalse
}
\newcommand\subsection@noopt[1]{%
  \global\@noindentafterheadingtrue
  \orig@subsection{#1}%
  \global\@noindentafterheadingfalse
}
\newcommand\subsection@star[1]{%
  \global\@noindentafterheadingtrue
  \orig@subsection*{#1}%
  \global\@noindentafterheadingfalse
}

\renewcommand\subsubsection{\@startsection{subsubsection}{3}{\z@}%
  {-2.0ex \@plus -.3ex \@minus -.2ex}
  {1.0ex \@plus .2ex}
  {\normalfont\bfseries}
}

\let\orig@subsubsection\subsubsection
\renewcommand\subsubsection{%
  \@ifstar{\subsubsection@star}{\subsubsection@nostar}%
}
\newcommand\subsubsection@nostar{%
  \@ifnextchar[{\subsubsection@opt}{\subsubsection@noopt}%
}
\newcommand\subsubsection@opt[2][]{%
  \global\@noindentafterheadingtrue
  \orig@subsubsection[#1]{#2}%
  \global\@noindentafterheadingfalse
}
\newcommand\subsubsection@noopt[1]{%
  \global\@noindentafterheadingtrue
  \orig@subsubsection{#1}%
  \global\@noindentafterheadingfalse
}
\newcommand\subsubsection@star[1]{%
  \global\@noindentafterheadingtrue
  \orig@subsubsection*{#1}%
  \global\@noindentafterheadingfalse
}

\newcommand{\inputnolabels}[1]{%
  {%
    \let\label\@gobble
    \input{#1}%
  }%
}

\makeatother

\newcommand{\mcB}{\mathcal{B}}
\newcommand{\mcC}{\mathcal{C}}
\newcommand{\mbG}{\mathbb{G}}
\newcommand{\msD}{\mathscr{D}}

\newcommand{\Z}{\mathbb{Z}}								
\newcommand{\Zz}{\mathbb{Z}_{\geq 0}}						
\newcommand{\Zp}{\mathbb{Z}_{> 0}}						
\newcommand{\Q}{\mathbb{Q}}								
\newcommand{\Qb}{\overline{\mathbb{Q}}}					
\newcommand{\C}{\mathbb{C}}								
\newcommand{\Ch}{\widehat{\mathbb{C}}}					
\newcommand{\Pj}{\mathbb{P}}								

\newcommand{\A}{\alpha}									%
\newcommand{\B}{\beta}									%
\newcommand{\G}{\gamma}								%
\newcommand{\la}{\lambda}								%

\newcommand{\OL}{\vspace{5mm}}							%
\newcommand{\HL}{\vspace{2mm}}							%
\newcommand{\q}{\quad}									

\newcommand{\f}[2]{\frac{#1}{#2}}							%
\newcommand{\npmod}[1]{\!\!\!\! \pmod{#1}}					%
\newcommand{\ceq}{\coloneqq} 							
\newcommand{\ie}{i.e.\ }									
\newcommand{\vt}[1]{\left\lvert #1 \right\rvert}					
\newcommand{\gen}[1]{\left\langle #1 \right\rangle}				
\newcommand{\fl}[1]{\left\lfloor #1 \right\rfloor}					
\newcommand{\onarrow}[1]{\overset{#1}{\longrightarrow}}			

\DeclareMathOperator{\id}{id}								
\DeclareMathOperator{\Aut}{Aut}							
\DeclareMathOperator{\Gal}{Gal}							
\DeclareMathOperator{\ord}{ord}							
\DeclareMathOperator{\Sym}{Sym}							

\pretocmd{\section}{\needspace{4\baselineskip}}{}{}
\pretocmd{\subsection}{\needspace{3\baselineskip}}{}{}
\pretocmd{\subsubsection}{\needspace{2\baselineskip}}{}{}

\newtheoremstyle{mythmstyle}  
  {12pt}   
  {12pt}   
  {\itshape} 
  {}      
  {\bfseries} 
  {.}     
  { }     
  {}      

\newtheoremstyle{mydefstyle}  
  {12pt}   
  {12pt}   
  {}       
  {}       
  {\bfseries} 
  {.}      
  { }      
  {}       

\newtheoremstyle{myremarkstyle} 
  {12pt}   
  {12pt}   
  {}       
  {}       
  {\normalfont} 
  {.}      
  { }      
  {}       

\theoremstyle{mythmstyle}
\newtheorem{lemma}{Lemma}[section]
\newtheorem{prop}[lemma]{Proposition}
\newtheorem{thm}[lemma]{Theorem}
\newtheorem{cor}[lemma]{Corollary}

\theoremstyle{mydefstyle}
\newtheorem{definition}[lemma]{Definition}
\newtheorem{exa}[lemma]{Example}
\newtheorem{rem}[lemma]{Remark}

\AtBeginEnvironment{thm}{\needspace{3\baselineskip}}
\AtBeginEnvironment{prop}{\needspace{3\baselineskip}}
\AtBeginEnvironment{lemma}{\needspace{3\baselineskip}}
\AtBeginEnvironment{cor}{\needspace{3\baselineskip}}
\AtBeginEnvironment{definition}{\needspace{3\baselineskip}}
\AtBeginEnvironment{exa}{\needspace{3\baselineskip}}

\newcommand{\thmref}[1]{Theorem~\ref{#1}}
\newcommand{\propref}[1]{Proposition~\ref{#1}}
\newcommand{\lemref}[1]{Lemma~\ref{#1}}
\newcommand{\corref}[1]{Corollary~\ref{#1}}
\newcommand{\defref}[1]{Definition~\ref{#1}}

\newcommand{\figref}[1]{Figure~\ref{#1}}
\newcommand{\tabref}[1]{Table~\ref{#1}}

\newcommand{\belyi}{Bely\u{\i}\xspace}
\newcommand{\dde}{dessin d'enfant\xspace}
\newcommand{\ddes}{dessins d'enfants\xspace}
\newcommand{\Dde}{Dessin d'enfant\xspace}

\newcommand{\DdEs}{Dessins d'Enfants\xspace}

\newcommand{\AD}{\Aut{\msD}}

\newcommand{\sigman}{(1\ 2\ \ldots\ n)}

\numberwithin{equation}{section}

\newcommand{\thmclassone}{Let $[n,n,n]$ be a uniform passport of genus at least~2.

Then it admits a nonregular \dde\ with nontrivial automorphism group
if and only if $n$ is composite.

Furthermore, if $n$ is composite, then for every divisor $r$ of $n$
satisfying $1<r<n$, there exists a dessin $\msD$ with passport $[n,n,n]$
such that $\vt{\AD}=r$.}

\newcommand{\thmclassthree}{If a uniform passport $[b^{q}, b^{q}, n]$ has genus at least~$2$, then it admits a \dde with a trivial automorphism group.

The same statement holds for the passports $[b^{q}, n, b^{q}]$ and $[n, b^{q}, b^{q}]$.}

\newcommand{\thmtrivautd}{
Let $\ell_{1}$ and $\ell_{2}$ be distinct prime numbers with
$\ell_{1} < \ell_{2}$.
Then every \dde with passport
$[\ell_{2}^{\ell_{1}}, \ell_{2}^{\ell_{1}}, \ell_{1}\ell_{2}]$
has trivial automorphism group.

The same conclusion holds for the passports
$[\ell_{1}\ell_{2}, \ell_{2}^{\ell_{1}}, \ell_{2}^{\ell_{1}}]$
and
$[\ell_{2}^{\ell_{1}}, \ell_{1}\ell_{2}, \ell_{2}^{\ell_{1}}]$.}

\begin{document}

\title{Automorphism Groups of Uniform Dessins d'Enfants of Genus at Least Two}

\author{Tatsuya Ohnishi}
\markboth{}{}

\begin{abstract}
For a smooth algebraic curve defined over a number field, one can associate a bipartite graph called a \emph{dessin d'enfant}.

We study the regularity and automorphism groups of dessins with \emph{uniform passports}. In a previous paper, we proved that every passport of the form $[n,b^{q},n]$ of genus at least~$2$ admits a dessin with trivial automorphism group. Here we prove the analogous result for passports of the form $[b^{q},b^{q},n]$.

We also construct examples of uniform passports of genus at least~2 for which every dessin with that passport has nontrivial automorphism group, and others for which every dessin with that passport has trivial automorphism group.

Finally, we give an alternative proof of the $[n,b^{q},n]$ case using counting arguments based on centralizers of permutations.
\end{abstract}

\maketitle

\vspace{-1.9\baselineskip}

\def\thefootnote{\fnsymbol{footnote}}
\makeatletter
\renewcommand\@makefntext[1]{%
  \noindent\hspace{0.5em}#1%
}
\makeatother
\footnotetext{Tatsuya~Ohnishi~(\raisebox{-0.18ex}{\includegraphics[height=2.05ex]{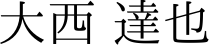}})}
\footnotetext{Graduate School of Information Science and Technology, The University of Osaka, Japan}
\footnotetext{e-mail: {\tt ohnishi-t@ist.osaka-u.ac.jp}}
\makeatletter
\renewcommand\@makefntext[1]{%
  \noindent\@makefnmark\ #1%
}
\makeatother

\def\thefootnote{\arabic{footnote}}

{\small
\tableofcontents
}



\section{Introduction}

\subsection{\belyi's Theorem and \DdEs}

\label{sec:belyi}
\begin{thm}[\belyi's Theorem]\cite{Belyi79}\cite{Belyi02}\cite[Theorem~1.3]{Jones16}
Let $X$ be a compact Riemann surface, that is, a smooth projective algebraic curve in
$\Pj_{\C}^{N}$ for some $N$.
Then $X$ can be defined over the field of algebraic numbers $\Qb$ if and only if there exists a non-constant meromorphic
function $\B\colon X \to \Ch\ (\ceq \C \cup \{ \infty \})$ ramified over at most three points.
\end{thm}

This deep theorem reveals a surprising connection between algebraic curves over
number fields and combinatorial structures on surfaces. This connection deeply
fascinated Grothendieck and led to his theory of \ddes.

For any compact Riemann surface $X$ defined over a number field, one can choose a suitable
function $\B$ (called a \emph{\belyi function}) such that all of its critical values lie among
three points.
By applying an appropriate M\"obius transformation, these points may be taken to be
$0$, $1$, and $\infty$.

For such a pair $(X, \B)$ (called a \emph{\belyi pair}), one can associate a bipartite graph
called a \emph{\dde} (child's drawing), or simply a \emph{dessin}
(see \defref{def:dessin}).

The dessin for $(X, \B)$ is drawn on an orientable surface of the same genus as $X$, where the black vertices ($\bullet$)
 and white vertices ($\circ$) represent $\B^{-1}(0)$ and $\B^{-1}(1)$, respectively\footnote{Some references
(such as \cite{Jones16}) represent $\B^{-1}(0)$ by white vertices and $\B^{-1}(1)$ by black vertices.
In this paper, we adopt the convention used in many classical references.}.
The faces --- that is, the connected components bounded by edges --- correspond to $\B^{-1}(\infty)$, and
the edges correspond to $\B^{-1}([0,1])$. Each face is homeomorphic to a disk.

By studying the properties of \ddes, one can combine insights from algebraic geometry and
combinatorics, including graph theory, to enrich both areas of analysis. This perspective also opens up
a range of possibilities for further applications.

Two dessins $\msD$ and $\msD'$ are said to be \emph{isomorphic} if there exists
an orientation-preserving homeomorphism between the underlying surfaces
that induces an isomorphism of the embedded bipartite graphs, preserving
vertex colors and the cyclic order of incident edges at each vertex.
Equivalently, if $(X,\B)$ and $(X',\B')$ are the corresponding \belyi pairs,
then $\msD$ and $\msD'$ are isomorphic if and only if there exists a
biholomorphic map $\varphi\colon X \to X'$ such that $\B'\circ\varphi=\B$.

Two examples of \ddes are shown in \figref{fig:dessins-ex}. Since $X=\Ch$ has genus~0 in both cases,
the corresponding dessins can be drawn on a sphere.
For the left dessin, we may take
\begin{align}
\B(z) = - \f{5z^{3}(3z-6+\sqrt{6})^{2}}{(9-4\sqrt{6})(-15z+12+2\sqrt{6})},
\end{align}
as a \belyi function.
The number of edges of a dessin is equal to the degree of the corresponding \belyi function, which is $5$ in this case.

\begin{figure}[htbp]
\centering
\includegraphics[width=135mm]{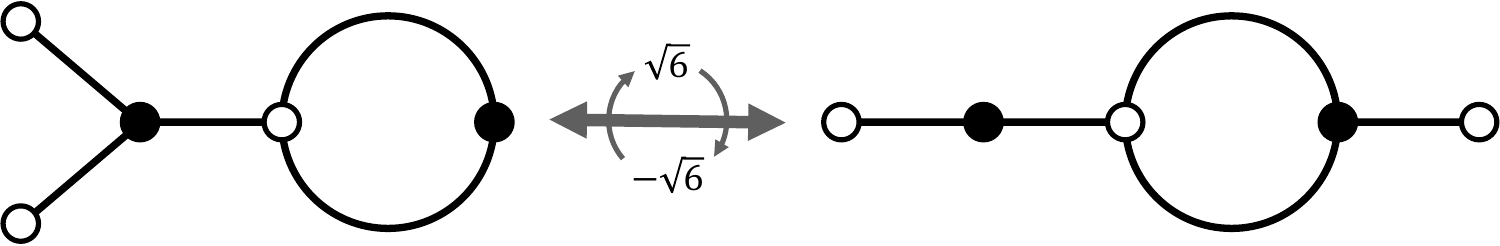} 
\caption{A pair of Galois-conjugate \ddes}
\label{fig:dessins-ex}
\end{figure}

By applying the Galois automorphism $\sqrt{6} \mapsto -\sqrt{6}$
to the coefficients of $\B(z)$, we obtain the dessin on the right. This is an example of
an action of the absolute Galois group $\mbG = \Gal(\Qb/\Q)$ on \ddes \cite[4.2.1]{Jones16}.

The action of $\mbG$ on the set of all dessins is faithful; that is,
for any two distinct elements of $\mbG$, there exists a dessin
whose images under these elements lie in distinct isomorphism classes.
Therefore, the study of Galois orbits of dessins provides a powerful tool for investigating the structure
of the absolute Galois group.
 
\subsection{Regularity and Automorphism Groups}

\label{sec:regaut}
Two fundamental invariants (up to isomorphism) of a \dde with respect to the Galois action are its regularity and its automorphism group.
A dessin is said to be \emph{regular} if its \emph{monodromy group} (see \defref{def:monog}) acts
regularly (that is, freely and transitively) on the set of edges.
From a geometric point of view, the \emph{automorphism group} of a dessin is the group of deck transformations of the
associated \belyi covering.
Equivalently, it can be identified with the centralizer of the monodromy group in the symmetric group acting on the edges.

By studying regularity and automorphism groups, one can gain insight into the symmetry properties of a \dde.
The order of the automorphism group always divides the number of edges of the dessin, and the equality of these two numbers is equivalent to the dessin being regular.
Thus, the structure of the automorphism group provides a precise measure of the symmetry exhibited by the dessin.

\figref{fig:regular} shows two dessins of genus~1 with the same passport; that is, they have the same valency list:
each has $8$ edges, two black vertices of valency $4$, four white vertices of valency $2$, and two faces of valency $4$.
Although both dessins may appear highly symmetric at first glance, the left dessin is regular, whereas the right one is not.
The automorphism group of the left dessin has order~$8$, while that of the right dessin has order~$4$.

\begin{figure}[htbp]
\centering
\includegraphics[width=135mm]{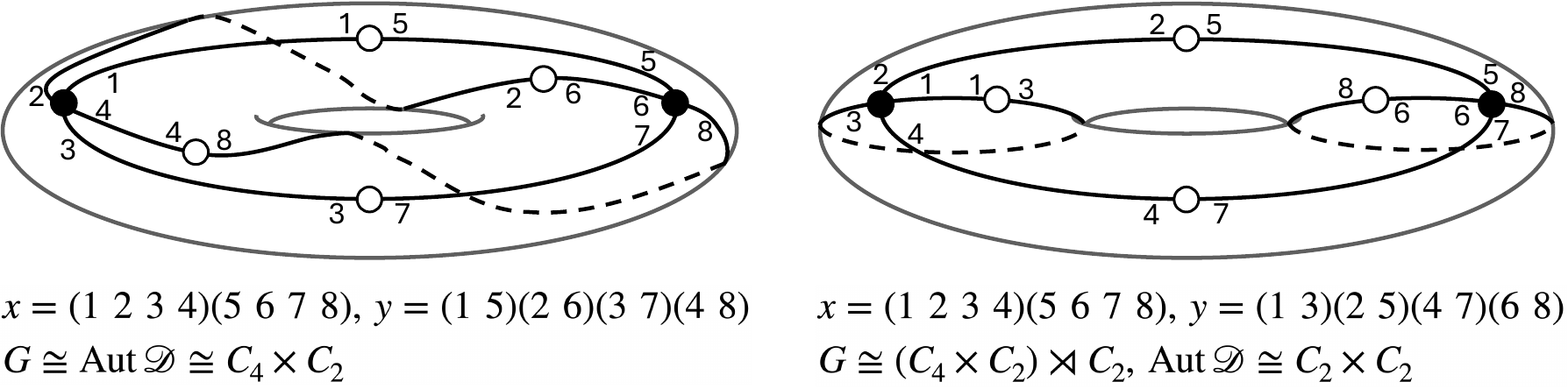} 
\caption{Regular and non-regular dessins with the same passport}
\label{fig:regular}
\end{figure}

Thus, regularity and automorphism groups may differ among dessins with the same passport.
However, they are invariant under the action of the absolute Galois group.
They therefore play an essential role in the study of families of dessins lying in the same Galois orbit.

Furthermore, regularity and automorphism groups have a significant impact on the relationship between
the field of moduli and the field of definition.

For a \dde $\msD$ corresponding to a \belyi pair $(X, \B)$, the \emph{field of moduli} $M(\msD)$
is defined as the fixed field of the subgroup
$G(\msD) = \{ \sigma \in \mbG \mid \msD \cong \msD^{\sigma} \} \le \mbG$.
That is, $M(\msD)$ is the subfield of $\Qb$ consisting of those elements that are fixed by every automorphism
$\sigma$ for which the conjugate dessin $\msD^{\sigma}$ is isomorphic to $\msD$. For the dessins in
\figref{fig:dessins-ex}, $M(\msD)$ is $\Q(\sqrt{6})$.

A number field $K$ is called a \emph{field of definition} of $\msD$ if both $X$ and $\B$ can be defined over $K$.
Unlike the field of moduli, a field of definition of $\msD$ is in general not unique,  and there need not exist
a smallest field of definition.

The field of moduli depends only on the isomorphism class of $\msD$ and is contained in any field of definition of $\msD$.
For most \ddes, the field of moduli is also a field of definition; that is, $X$ and $\B$ can be defined over $M(\msD)$.
However, this is not always the case.
In such situations, there is no field of definition of $\msD$ that is fixed by all $\sigma \in G(\msD)$.

This remarkable phenomenon does not occur for dessins with trivial automorphism groups. Interestingly, it does not
occur for regular dessins either. Thus, in the context of this problem, the most interesting cases lie between these
two extremes, namely nonregular dessins with nontrivial automorphism groups.

In this paper, we develop methods for analyzing how regular dessins, dessins with trivial automorphism groups,
and nonregular dessins with nontrivial automorphism groups are distributed. These results provide new information
on the interplay between regularity and automorphism groups, and may help clarify their role in the study of fields of
moduli and fields of definition.

\subsection{Uniform Passports and Dessins}

\label{sec:uniform}
A \dde is said to be \emph{uniform} if the valencies of black vertices, white vertices, and faces are each constant.
One also says that it has a uniform passport (or uniform valency list).

Uniform dessins exhibit a high degree of symmetry. However, while every regular dessin is uniform, the converse does not hold: a uniform dessin need not be regular. In other words, uniformity does not represent the highest possible level of symmetry.

As an example, consider uniform passports $[a^{p}, b^{q}, c^{r}]$, where $n = pa = bq = rc$, and suppose that $n = 6$.
By symmetry among black vertices, white vertices, and faces, we may assume $c \ge a \ge b$ (equivalently, $r \le p \le q$).

In genus~$0$, the uniform passports are $[6, 1^{6}, 6]$ and $[2^{3}, 2^{3}, 3^{2}]$.
The corresponding dessins are shown in \figref{fig:genus0-n6}.
In both cases, the dessins are regular.

\begin{figure}[htbp]
\centering
\includegraphics[width=60mm]{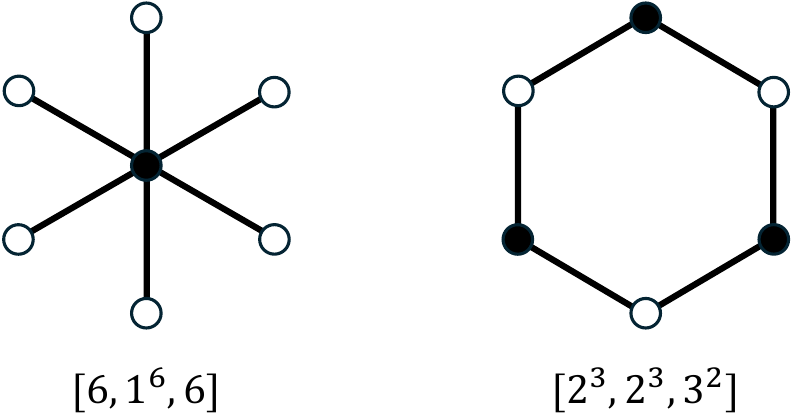}
\caption{Uniform \ddes of genus~0, degree 6}
\label{fig:genus0-n6}
\end{figure}

In genus~1, the uniform passports are $[3^{2}, 2^{3}, 6]$ and $[3^{2}, 3^{2}, 3^{2}]$, and the corresponding dessins are shown in \figref{fig:genus1-n6}.
The dessin on the left is regular, whereas the one on the right is not.
The automorphism group of the right-hand dessin has order~2, which is strictly smaller than 6.

\begin{figure}[htbp]
\centering
\includegraphics[width=135mm]{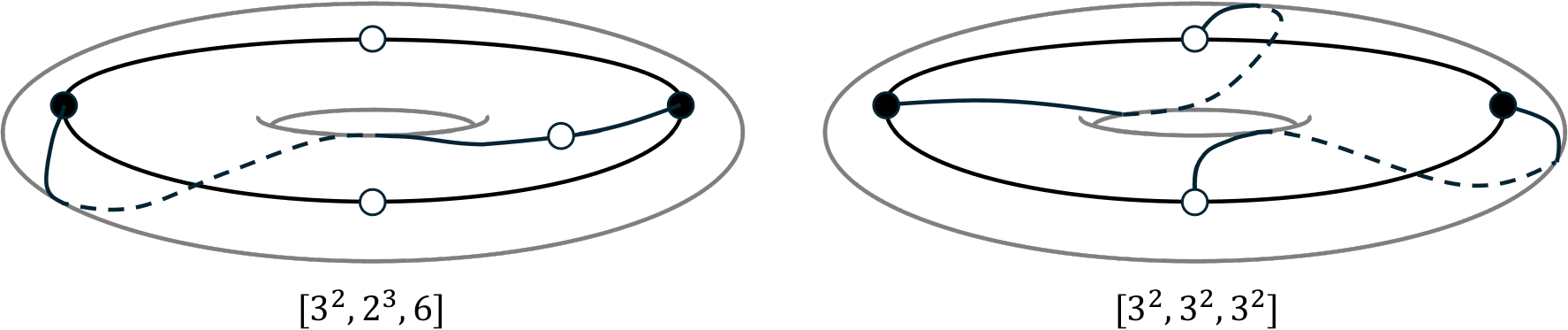}
\caption{Uniform \ddes of genus~1, degree 6}
\label{fig:genus1-n6}
\end{figure}

For genus at least~2, the only uniform passport is $[6, 3^{2}, 6]$, which has genus~2.
There are four dessins with this passport, shown in \figref{fig:genus2-n6}.
The orders of their automorphism groups are 6, 3, 2, and 1, respectively, from the upper left to the lower right.
Only the upper-left dessin is regular; the others are non-regular, and the lower-right dessin has a trivial automorphism group.

\begin{figure}[htbp]
\centering
\includegraphics[width=135mm]{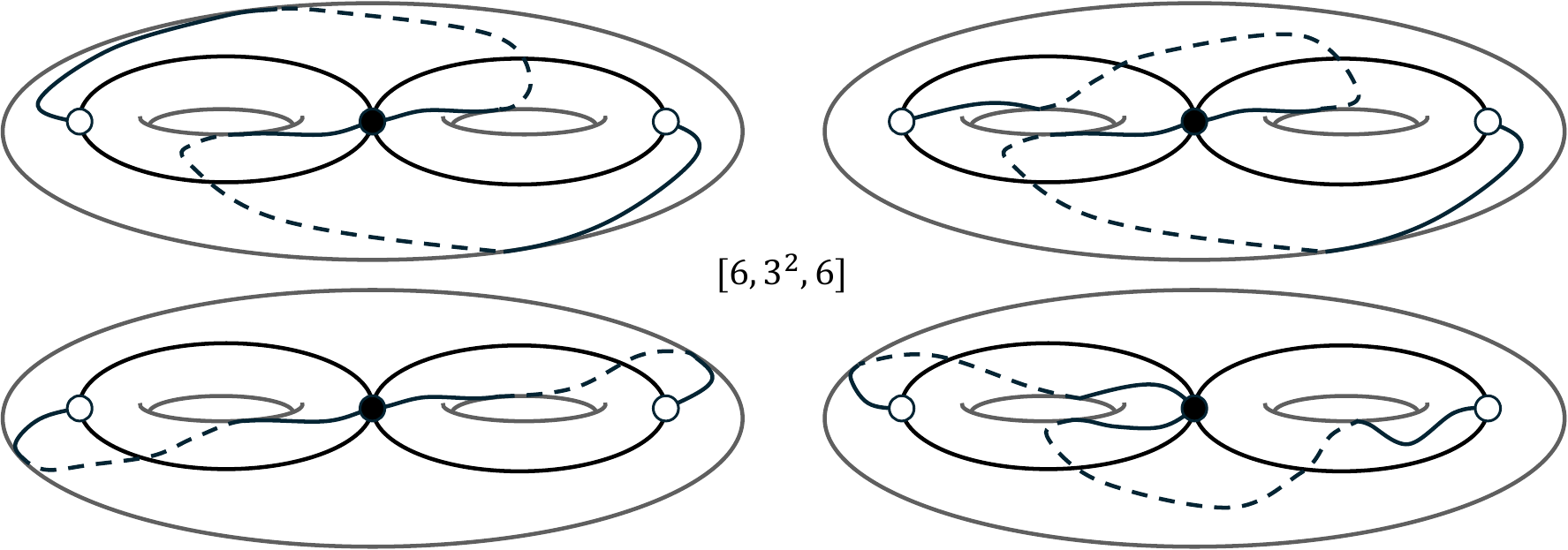}
\caption{Uniform \ddes of genus~2, degree 6}
\label{fig:genus2-n6}
\end{figure}

\subsection{Previous Work}

Examination of the regularity and automorphism groups of uniform \ddes suggests that,
as the genus increases, the proportion of regular dessins tends to decrease,
whereas the proportion of dessins with trivial automorphism groups tends to increase.

For a given uniform passport, the following questions naturally arise:
\begin{itemize}
\item Under what conditions do regular \ddes exist, and how many are there?
\item How are the automorphism groups distributed?
\item How does this distribution change with the genus?
\end{itemize}
\HL

To the best of our knowledge, relatively little attention has been paid to the study of uniform passports as 
families of dessins and to the systematic investigation of their regularity and automorphism groups.
In our previous paper \cite{Ohnishi26}, we proposed the conjectural picture shown in \tabref{tab:genusautd} and
established several supporting results.

\begin{table}[htbp]
\centering
\begin{tabular}{|l|c|c|c|}
\hline
Genus & $\Aut \msD \cong \{ 1 \}$ & $1 < \lvert \AD \rvert < n$ & $\lvert \AD \rvert = n$ (regular) \\
\hline
0 & -- \textsuperscript{\ref{itm:g0reg}} & -- \textsuperscript{\ref{itm:g0reg}} & $\checkmark$\textsuperscript{\ref{itm:g0reg}} \\ \hline
1 (with an $(n)$-cycle) & -- \textsuperscript{\ref{itm:g1triv}}  & -- \textsuperscript{\ref{itm:g1nreg}} & $\checkmark$\textsuperscript{\ref{itm:anyg}} \\ \hline
1 (without an $(n)$-cycle) & -- \textsuperscript{\ref{itm:g1triv}}  & $\checkmark$\textsuperscript{\ref{itm:g1nreg}} & $\diamond$ \\ \hline
$\ge 2$ (with $\ge 2$ $(n)$-cycles) & $\checkmark$\textsuperscript{\ref{itm:g2triv}} & $\diamond$ & $\checkmark$\textsuperscript{\ref{itm:anyg}} \\ \hline
$\ge 2$ (with $\le 1$ $(n)$-cycle) & $\checkmark$* & $\checkmark$* & $\diamond$ \\
\hline
\end{tabular}\\
\begin{flushleft}
\hspace{16pt}$n$: number of edges of the dessin d'enfant \\
\hspace{16pt}$\checkmark$: always occurs,\quad --: never occurs,\quad $\diamond$: depends on the passport.
\end{flushleft}
\HL
\caption{Distribution of automorphism groups for uniform passports (Previous)}
\label{tab:genusautd}
\end{table}

Asterisks (*) indicate statements that were conjectural at the stage of \cite{Ohnishi26}.
The numbers in parentheses refer to the corresponding items listed below.

\begin{enumerate}
\item\label{itm:g0reg} Genus~0: all uniform dessins are regular and have automorphism groups of order~$n$,
isomorphic to their monodromy groups.
\item\label{itm:g1nreg} Genus~1: a uniform passport admits a non-regular dessin if and only if it does not contain
an element of type $n^{1}$ (that is, it does not correspond to a tree).
\item\label{itm:g1triv} Genus~1: no uniform passport admits a dessin with a trivial automorphism group.
\item\label{itm:g2triv} Genus $\ge 2$: every uniform passport admits a dessin with a trivial automorphism group, and hence admits
a non-regular dessin.
Moreover, if the passport contains at most one cycle of type $(n)$, then it also admits a non-regular dessin with
a nontrivial automorphism group.
\item\label{itm:anyg} Any genus: a passport of the form $[a^{p}, b^{q}, n]$, or any permutation thereof in the tree case, admits a regular dessin if and only if $\gcd(p, q) = 1$.
\end{enumerate}

In \cite{Ohnishi26} we proved \ref{itm:g0reg}, \ref{itm:g1nreg}, \ref{itm:g1triv}, and \ref{itm:anyg},
and partially proved \ref{itm:g2triv} for the case of $[n, b^{q}, n]$ with $q \ge 1$.

These results clarify several aspects of the regularity and automorphism groups of uniform dessins.
They describe the distribution of automorphism groups in genera~0 and~1, establish a necessary and sufficient
condition for a uniform passport in the tree case  (where the passport contains $n^{1}$) to admit a regular dessin,
and provide a group-theoretic characterization of regularity for general uniform passports. For genus at least~2,
and especially for the tree case, we developed a new method for proving the existence of dessins with trivial
automorphism groups by estimating the number of relevant permutations. Combined with an analysis of the
primitivity of the monodromy group, this method yields the desired existence theorem.

\subsection{Main Results}

In this paper, we provide the following results.

\begin{enumerate}[label=(\alph*)]
\item \label{itm:g2c3triv} Extending the result corresponding to \ref{itm:g2triv} in \tabref{tab:genusautd},
we prove the following theorem in Section \ref{sec:class3}.
\end{enumerate}
\textbf{\thmref{thm:class3}.}
\textit{\thmclassthree}
\HL

For the conjectural parts of \tabref{tab:genusautd}, we found several counterexamples.
We discuss the following in Section \ref{sec:counterex}.
\begin{enumerate}[label=(\alph*),resume]
\item \label{itm:g2c8notriv}
The passport $[8^{2}, 2^{8}, 4^{4}]$ admits no \dde\ with trivial automorphism group.
This provides a counterexample to the corresponding conjecture in \tabref{tab:genusautd}
and reveals a new family of passports whose automorphism groups exhibit unexpected behavior.
\item \label{itm:g2c3onlytriv}
Conversely, the passport $[5^{3}, 5^{3}, 15]$ admits no \dde\ with a nontrivial automorphism group.
This provides a counterexample to the corresponding conjecture in \tabref{tab:genusautd}.
Generalizing this observation, we prove the following theorem.
\end{enumerate}
\textbf{\thmref{thm:trivautd}.}
\textit{\thmtrivautd}
\HL

\begin{enumerate}[label=(\alph*),resume]
\item \label{itm:g2c1} For passports of the form $[n, n, n]$ of genus at least~2, we prove the following theorem
in Section \ref{sec:class1}.
\end{enumerate}
\textbf{\thmref{thm:class1}.}
\textit{\thmclassone}
\HL

\begin{rem}
\cite{Horie24} gives a method for computing the number of isomorphism classes of dessins with two vertices
(one black and one white) and $L$ faces whose automorphism group has order $r$, where $r$ is a divisor of the
number of edges. The passport $[n,n,n]$ corresponds to the case $L=1$.

In this paper, we provide
an explicit construction of a dessin whose automorphism group has order $r$.
\end{rem}

Moreover,
\begin{enumerate}[label=(\alph*),resume]
\item \label{itm:g2c2triv}
In Section \ref{sec:class2}, using the methods developed in the proof of
\thmref{thm:trivautd}, we give an alternative proof of
\cite[Theorem~7.6]{Ohnishi26}, which states that every uniform passport
$[n,b^{q},n]$ of genus at least~2 admits a \dde\ with a trivial automorphism group.
\end{enumerate}

The current results and conjectures for uniform passports of genus at least~2 are summarized in \tabref{tab:genusautd2}.
Asterisks (*) indicate conjectural statements.
The letters in parentheses indicate the corresponding items listed above.

\begin{table}[htbp]
  \centering
\begin{tabular}{|l|c|c|c|}
\hline
Passport & $\Aut \msD \cong \{ 1 \}$ & $1 < \lvert \AD \rvert < n$ & $\lvert \AD \rvert = n$ (regular) \\
\hline
$[n, n, n]$ ($n$: prime) & $\checkmark$ & --\textsuperscript{\ref{itm:g2c1}} & $\checkmark$ \\ \hline
$[n, n, n]$ ($n$: composite) & $\checkmark$ & $\checkmark$\textsuperscript{\ref{itm:g2c1}} & $\checkmark$ \\ \hline
$[n, b^{q}, n]$ ($q \ge 2$) & $\checkmark$\textsuperscript{\ref{itm:g2c2triv}} & $\checkmark$* & $\checkmark$ \\ \hline
$[b^{q}, b^{q}, n]$ ($q \ge 2$) & $\checkmark$\textsuperscript{\ref{itm:g2c3triv}} & $\diamond$\textsuperscript{\ref{itm:g2c3onlytriv}} & -- \\ \hline
$[a^{p}, b^{q}, n]$ ($q > p \ge 2$) & $\checkmark$* & $\checkmark$* & \makecell{$\checkmark$ ($\gcd(p,q)=1$)\\--\,\, ($\gcd(p,q)\ne 1$)} \\ \hline
$[a^{p}, b^{q}, c^{r}]$ ($p, q, r \ge 2$) & $\diamond$\textsuperscript{\ref{itm:g2c8notriv}} & $\checkmark$* & $\diamond$ \\
\hline
\end{tabular}\\
\begin{flushleft}
\hspace{22pt}$n$: number of edges of the dessin d'enfant \\
\hspace{22pt}$\checkmark$: always occurs,\quad --: never occurs,\quad $\diamond$: depends on the passport.
\end{flushleft}
\HL
\caption{Distribution of automorphism groups for uniform passports (Genus $\ge 2$)}
\label{tab:genusautd2}
\end{table}
\vspace{-\baselineskip}

For future work, two directions seem particularly promising.

The first is to establish the conjectural statements summarized in
\tabref{tab:genusautd2}. In particular, it would be desirable to prove the
existence of dessins with trivial automorphism groups for passports of the
form $[a^{p}, b^{q}, n]$ ($q > p \ge 2$) of genus at least~$2$, and to determine conditions
under which a passport $[a^{p}, b^{q}, c^{r}]$ $(p,q,r\ge2)$ admits such a dessin.

The second is to obtain quantitative results on the distribution of regular
dessins and automorphism groups.

\section{Preliminaries}

The definitions and results in this chapter are based primarily on \cite{Jones16} and \cite{Adrianov20}.

\begin{definition}[\Dde]\cite[Definition~2]{Jones16}
\label{def:dessin}
A \emph{\dde}, or simply a \emph{dessin}, is a map consisting of a connected, finite, bipartite graph embedded
in a connected, compact, oriented surface without boundary.
Here, a bipartite graph is a graph whose vertices can be colored black and white in such a way that each edge
joins a black vertex to a white vertex.
\end{definition}

Since a compact Riemann surface provides a suitable surface on which a \dde can be embedded,
one can draw a dessin corresponding to a \belyi pair $(X, \B)$, as described in
Section~\ref{sec:belyi}.

\HL

A \emph{partition} $\la$ of a positive integer $n$, denoted by $\la \vdash n$, is a multiset of
positive integers whose sum is $n$, where the order of the parts is irrelevant.
 
\begin{definition}[Passport of a dessin]\cite[Definition~2.10]{Adrianov20}
\label{def:passport}
Let $n$ be the number of edges of a \dde.
The triple $[\lambda_{0}, \lambda_{1}, \lambda_{\infty}]$ of partitions $\lambda_{0}, \lambda_{1}, \lambda_{\infty} \vdash n$,
which correspond respectively to the valencies of the black vertices, the white vertices, and
the faces of the dessin, is called a $\emph{passport}$ of the dessin.
\end{definition}

Both dessins in \figref{fig:dessins-ex} have passport $[32,311,41]$,
while the dessins in \figref{fig:regular} have passport
$[4^{2},2^{4},4^{2}]$.
Throughout this paper, we use the standard abbreviated notation for partitions:
for example, $32=(3,2)$, $311=(3,1,1)$, $41=(4,1)$, $4^{2}=(4,4)$, and $2^{4}=(2,2,2,2)$.

Since $\lambda_{0}, \lambda_{1}, \lambda_{\infty}$ are partitions of $n$, we have

\begin{align}
\label{eq:lambdan}
\lvert \lambda_{0} \rvert = \lvert \lambda_{1} \rvert = \lvert \lambda_{\infty} \rvert = n,
\end{align}
where $\lvert \la \rvert$ denotes the sum of the parts of the partition $\la$. 

The numbers of vertices, faces, and edges are $l(\lambda_{0}) + l(\lambda_{1})$,
$l(\lambda_{\infty})$, and $n$, respectively, where $l(\la)$ denotes the number of parts
of the partition $\la$.
Therefore, the genus $g$ of the underlying curve $X$ satisfies

\begin{align}
l(\lambda_{0}) + l(\lambda_{1}) + l(\lambda_{\infty}) - n = 2 - 2g,
\end{align}
and hence

\begin{align}
\label{eq:lambdag}
g = \f{n - (l(\lambda_{0}) + l(\lambda_{1}) + l(\lambda_{\infty}))}{2} + 1.
\end{align}
Since $g$ is a non-negative integer, it follows that

\begin{align}
\label{eq:lambdal}
l(\lambda_{0}) + l(\lambda_{1}) + l(\lambda_{\infty}) \le  n + 2, \q
l(\lambda_{0}) + l(\lambda_{1}) + l(\lambda_{\infty}) \equiv n \npmod{2}.
\end{align}

Note that not every triple of partitions satisfying \eqref{eq:lambdan} and \eqref{eq:lambdal}
is realized by a \dde.
For example, although $[2^{2}, 2^{2}, 31]$ formally satisfies these conditions and would yield genus~$0$,
there exists no dessin with this passport.
Similarly, the passport $[3^{2}, 3^{2}, 42]$, which would correspond to genus~$1$, admits no dessin.

These facts can be proved by showing that there exists no corresponding monodromy group
(see \defref{def:monog}) in each case.

\begin{definition}[Uniform passports and dessins]\cite[Remark~3.2]{Jones16}
The passport of a \dde given by
$[\lambda_{0}, \lambda_{1}, \lambda_{\infty}] = [a_{1}\cdots a_{p},\ b_{1}\cdots b_{q},\ c_{1}\cdots c_{r}]$
is called \emph{uniform} if
\begin{align}
a_{1} = \cdots = a_{p}, \quad b_{1} = \cdots = b_{q}, \quad c_{1} = \cdots = c_{r},
\end{align}
\ie if the passport takes the form $[a^{p}, b^{q}, c^{r}]$.

A \dde is called \emph{uniform} if it has a uniform passport.
\end{definition}

For a uniform passport $[a^{p}, b^{q}, c^{r}]$ with $n = pa = qb = rc$ and the genus~$g$ of the
underlying curve $X$ corresponding to a dessin with this passport, by \eqref{eq:lambdan} we have

\begin{align}
\label{eq:genus-g}
p + q + r - n = 2 - 2g,
\end{align}
and hence

\begin{align}
\label{eq:genus}
g = \f{n-(p+q+r)}{2} + 1.
\end{align}

\begin{definition}[Monodromy group of a \dde]\cite[2.1.1]{Jones16}
\label{def:monog}
Define two permutations $x$ and $y$ acting on the set of edges $E$ of a \dde $\msD$ as follows.
For each edge $e \in E$, define $x \cdot e$ and $y \cdot e$ to be the next edges around the unique black vertex and
the unique white vertex incident to $e$, respectively, following the counterclockwise orientation.

The \emph{monodromy group} of $\msD$ is the subgroup $G = \gen{x, y}$ generated by $x$ and $y$ in
the symmetric group $\Sym(E)$ of all permutations of $E$.
\end{definition}

Since a \dde is a connected graph, it follows that any edge in $E$ can be mapped to any other edge by the action of $G$.
Therefore, the monodromy group $G$ acts transitively on $E$.

\begin{rem}
Throughout this paper, permutations act on the left, and products are composed from right to left,
\ie $(xy) \cdot e = x \cdot (y \cdot e)$ for an edge $e$.
\end{rem}

Another important observation concerning the monodromy group is that, in addition to $x$ and $y$ encoding
the cycles of the black and white vertices, respectively, the permutation $z = (xy)^{-1}$ encodes the cycles
 corresponding to the faces.
In fact, for each face, half of the edges incident to it form a cycle of $z$, while the remaining edges belong to cycles of
$z$ corresponding to the neighboring faces.

It is known that if a group $G$ generated by two elements acts transitively on a set of
edges $E$, then there exists a \dde whose monodromy group is isomorphic to $G$\cite[Theorem~3.6]{Scodro24}.

Therefore, studying groups that act transitively is essential for investigating the properties of \ddes and, consequently, of algebraic curves.

\begin{definition}[Regular dessins]\cite[2.1.2]{Jones16}
\label{def:regular}
A \dde is called \emph{regular} if its monodromy group acts freely (that is, semiregularly) on the set of its edges.
\end{definition}

This implies that the monodromy group of a regular dessin acts freely and transitively --- hence regularly ---
on its edges.

The following criterion relates the order of the monodromy group to the number of edges.

\begin{lemma}
\label{lem:order-n}
A \dde with $n$ edges is regular if and only if the order of its monodromy group is $n$.
\end{lemma}

\begin{proof}
See \cite[Lemma~2.7]{Ohnishi26}.
\end{proof}

\vspace{-9pt}
\nopagebreak

\begin{definition}[Automorphism group of a \dde]\cite[2.1.2]{Jones16}
For a \dde $\msD$, we define its \emph{automorphism} to be a permutation of the set $E$ of $\msD$ which preserves the
cyclic order of edges around each vertex, that is, which commutes with $x$ and $y$, or equivalently, commutes with $G$.
Thus we can define an \emph{automorphism group} of $G$ as the centralizer:

\begin{align}
\Aut \msD \ceq C_{\Sym(E)}(G) &= \{ c \in \Sym(E) \mid cg = gc\ \text{for all}\ g \in G \} \\
\label{eq:cxcy}
&= \{ c \in \Sym(E) \mid cx = xc,\ cy = yc \},
\end{align}
where $\Sym(E)$ is the symmetric group of all permutations of $E$.

When $\Aut \msD \cong \{1 \}$, it is said that $\msD$ has a \emph{trivial} automorphism group.
\end{definition}

Since $\gen{x, y} = \gen{x, z} = \gen{y, z}$, where $z = (xy)^{-1}$,
both the monodromy group and the automorphism group are invariant under any permutation of black vertices,
white vertices, and faces.

The automorphism group of a \dde has the following properties:

\begin{itemize}
\item $\AD$ acts freely on the edges of $\msD$.
\item $\lvert \AD \rvert$ divides the number of edges.
\item If $\lvert \AD \rvert$ equals the number of edges, then $\AD \cong G$.
\end{itemize}

The following proposition is a basic result describing the relationship between regularity and automorphism groups.

\begin{prop}
\label{prop:regaut}
A \dde $\msD$ is regular if and only if its monodromy group $G$ is isomorphic to $\AD$.
\end{prop}

\begin{proof}
See \cite[Theorem~2.1]{Jones16}.
\end{proof}

This also implies that a \dde $\msD$ is regular if and only if $\lvert \AD \rvert$ equals the number of edges.

Note that regularity does not imply $G = \Aut \msD$;
these groups act on the edges of the dessin in different ways.

\begin{prop}
\label{prop:reguni}
If a \dde is regular, then it has a uniform passport.
\end{prop}

\begin{proof}
See \cite[Proposition~4.42]{Girondo12} and \cite[Proposition~2.10]{Ohnishi26}.
\end{proof}

The converse of this proposition does not hold in general.
This shows that regularity exhibits a higher degree of symmetry than uniformity.

\begin{exa}
The uniform passport $[4^{2}, 2^{4}, 4^{2}]$ (genus~1) corresponds to the two dessins shown in
Section~\ref{sec:regaut}, \figref{fig:regular}.
The dessin on the left is regular, and both its monodromy group and its automorphism group are
isomorphic to $C_{4} \times C_{2}$, the direct product of cyclic groups of orders 4 and 2, respectively.
In contrast, the dessin on the right is not regular; its monodromy group is isomorphic to $(C_{4} \times C_{2}) \rtimes C_{2}$
and has order~16, whereas its automorphism group has order~4 and is isomorphic to $C_{2} \times C_{2}$.
\end{exa}

\section{Counting Arguments for Trivial Automorphism Groups}
\label{sec:estimation}

\subsection{Settings}

A passport $[\la_{0}, \la_{1}, \la_{\infty}]$ is called a \emph{tree}
if at least one of the $\la_{i}$ is $n^{1}$. A \dde is called a tree if its passport is a tree\footnote{Strictly speaking, a dessin
is a tree if it has a single face. However, due to the natural $S_{3}$-symmetry permuting black vertices, white vertices,
and faces, we use the term tree passport for any passport containing $n^{1}$.}.

In this section, we consider uniform tree passports $[n, b^{q}, a^{p}]$ with $n = pa = qb$.
That is, we assume that $x$, $y$, and $z=(xy)^{-1}$ have cycle types
$(n)$, $(b^{q})$, and $(a^{p})$, respectively.

Since the monodromy group and the automorphism group are invariant under permutations of
$\la_{0}$, $\la_{1}$, and $\la_{\infty}$, it suffices to consider passports of this form;
the case $[a^{p}, b^{q}, n]$ follows by symmetry.

Let $S_n$ be the symmetric group on $\{1,\ldots,n\}$,
and fix an $n$-cycle $x\in S_n$, for example $x=\sigman$.
We define the following subsets of $S_n$:
\begin{align}
\label{eq:tncd}
\begin{aligned}
&T(b,q) \ceq \{ y\in S_{n} \mid y \text{ has cycle type } (b^{q}) \}, \\
&N(b,q,a) \ceq \{ y \in T(b,q) \mid (xy)^{-1} \text{ has cycle type } (a^{p}) \}, \\
&C(b,q) \ceq \{ y \in T(b,q) \mid \AD \ncong \{1\} \}, \\
&D(n) \ceq \{ y \in S_{n} \mid \AD \ncong \{1\} \},
\end{aligned}
\end{align}
where $\msD$ denotes the dessin corresponding to the monodromy group $\gen{x, y}$.

\figref{fig:tncd} illustrates the relationships among these sets.
The hatched region corresponds to $C(b, q)$.

\begin{figure}[htbp]
\centering
\includegraphics[width=66mm]{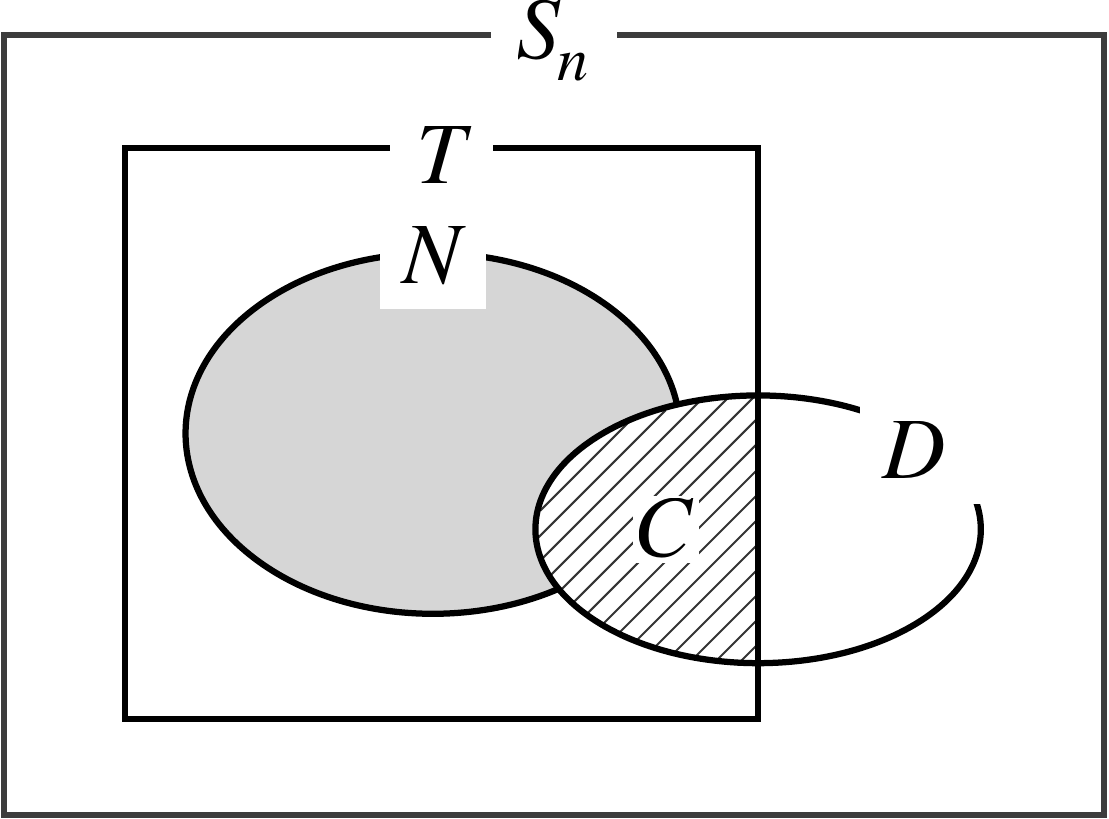} 
\caption{The sets related to the passport $[n, b^{q}, a^{p}]$}
\label{fig:tncd}
\end{figure}

The existence of a dessin $\msD$ with passport $[n, b^{q}, a^{p}]$ and trivial automorphism group
is equivalent to
\begin{align}
N \cap C^{c} \ne \emptyset
\end{align}
(the gray region in \figref{fig:tncd}). Since $C  = D \cap T$ and $N \subset T$, it is also equivalent to
\begin{align}
N \cap D^{c} \ne \emptyset.
\end{align}
Thus, to prove the existence of such a dessin,
it suffices to show that
\begin{align}
\vt{N} > \vt{D},
\end{align}
or equivalently,
\begin{align}
\f{\vt{N}}{\vt{T}} > \f{\vt{D}}{\vt{T}}.
\end{align}

\begin{prop}
\label{prop:tbq}
The number of elements in $S_{n}$ of cycle type $(b^{q})$ $(n = bq)$ is given by

\begin{align}
\vt{T(b, q)} = \f{n!}{b^{q}q!}.
\end{align}
\end{prop}

\begin{proof}
See \cite[Proposition~3.1]{Ohnishi26}.
\end{proof}
\OL

\subsection{Lower Bound for \textit{N}}

A useful tool for computing the values of $\vt{N}$ is the following theorem from \cite{Goupil98}.

\begin{thm}
\label{thm:goupil}
Let $\lambda = (\lambda_{1}, \dotsc, \lambda_{l})$ and $\mu = (\mu_{1}, \dotsc, \mu_{m})$
be partitions of $n$.
Define the genus associated with the pair $(\lambda, \mu)$ by

\begin{align}
g = \frac{n-(l+m)+1}{2},
\end{align}
and assume that $g \in \Zz$.

Let $C_\lambda$ and $C_\mu$ denote the conjugacy classes in $S_n$
consisting of permutations of cycle types $\lambda$ and $\mu$,
respectively. 
Let $c_{\lambda \mu}^{n}$ denote the number of solutions $(\sigma, \rho) \in C_{\lambda} \times C_{\mu}$
to the equation $\sigma \rho = \pi$,
where $\pi$ is a fixed $n$-cycle in $S_{n}$.
Then $c_{\lambda \mu}^{n}$ is given by

\begin{align}
c_{\la \mu}^{n} = \f{n}{z_{\la}z_{\mu}2^{2g}}\sum_{\substack{g_{1},g_{2}\ge 0\\ g_{1}+g_{2}=g}}(l+2g_{1}-1)! (m+2g_{2}-1)!
\sum_{\substack{(i_{1},\dotsc,i_{l})\vDash g_{1}\\(j_{1},\dotsc,j_{m})\vDash g_{2}}}
\prod_{k=1}^{l}\binom{\la_{k}}{2i_{k}+1}\prod_{k=1}^{m}\binom{\mu_{k}}{2j_{k}+1}.
\end{align}
Here $p \vDash n$ denotes a composition of $n$, that is, a finite sequence of non-negative integers
summing to $n$, where the order of the terms matters. We adopt the convention that $\binom{a}{b}=0$ if $b>a$.
Moreover, for a partition $\lambda = 1^{\A_{1}}\cdots n^{\A_{n}}$, where $\A_{i}$ denotes the multiplicity of $i$, we define
$z_{\lambda} = \prod_{i} \A_{i}! \, i^{\A_{i}}$.
\end{thm}

\begin{proof}
See \cite[Theorem~2.1]{Goupil98}.
\end{proof}

Using this theorem, we proved the following theorem:

\begin{thm}
\label{thm:MN}
Let $n \ge 3$ be an integer such that $n = bq$, with $b \ge 2$, and assume that $n \equiv q \pmod{2}$.
Fix $x = \sigman \in S_{n}$. Then

\begin{align}
\f{N(b, q, n)}{T(b, q)} \ge \f{2}{n+2}.
\end{align}
In particular, the inequality becomes an equality when $b = 2$.
\end{thm}

\begin{proof}
See \cite[Theorem~3.3]{Ohnishi26}.
\end{proof}

\begin{lemma}
\label{lem:Nbqblower}
Let $[n,b^{q},b^{q}]$ be a passport with $n=qb$, and suppose that it has genus at least~$1$.
Fix an $n$-cycle $x \in S_n$, and let $N(b,q,b)$
denote the set of permutations $y \in S_n$ of cycle type $(b^{q})$ such that $(xy)^{-1}$ also has cycle type $(b^{q})$.
Then
\begin{align}
\vt{N(b, q, b)} \ge \f{(n-q)!}{2^{q(b-2)}\left(\f{q+1}{2}\right)!\left(\f{q-1}{2}\right)!}\left(\f{b-1}{2b}\right)^{\f{q-1}{2}}.
\end{align}
\end{lemma}

\begin{proof}
Since $x^{-1}$ also has cycle type $(n)$, \thmref{thm:goupil} implies that
$c_{(b^{q}),(b^{q})}^{n}$ is the number of pairs $(\sigma,\rho)$ such that
both $\sigma$ and $\rho$ have cycle type $(b^{q})$ and satisfy $\sigma\rho = x^{-1}$.
Moreover,
\begin{align}
\sigma\rho = x^{-1} \iff \rho = (x\sigma)^{-1}.
\end{align}
Hence $c_{(b^{q}),(b^{q})}^{n}$ is precisely the number of permutations
$\sigma$ of cycle type $(b^{q})$ such that $(x\sigma)^{-1}$ also has cycle type $(b^{q})$,
namely $\vt{N(b,q,b)}$.

Therefore, by \thmref{thm:goupil},
\begin{align}
\lvert N&(b, q, b) \rvert = c_{(b^{q})(b^{q})}^{n} \\
&= \f{n}{2^{2g}b^{2q}(q!)^{2}}\sum_{\substack{g_{1},g_{2}\ge 0\\ g_{1}+g_{2}=g}}(q+2g_{1}-1)! (q+2g_{2}-1)!
\sum_{\substack{(i_{1},\dotsc,i_{q})\vDash g_{1}\\(j_{1},\dotsc,j_{q})\vDash g_{2}}}
\prod_{k=1}^{q}\binom{b}{2i_{k}+1}\prod_{k=1}^{q}\binom{b}{2j_{k}+1},
\end{align}
where
\begin{align}
g = \f{q(b-2)+1}{2}.
\end{align}

Since $g\ge1$, the outer sum contains the two distinct terms
corresponding to $(g_{1},g_{2})=(0,g)$ and $(g,0)$.
Hence
\begin{align}
\lvert N&(b, q, b) \rvert \\
&\ge  \f{n}{2^{2g}b^{2q}(q!)^{2}}\sum_{\substack{(g_{1}, g_{2})=\\(0, g), (g, 0)}}(q+2g_{1}-1)! (q+2g_{2}-1)! 
\sum_{\substack{(i_{1},\dotsc,i_{q})\vDash g_{1}\\(j_{1},\dotsc,j_{q})\vDash g_{2}}}
\prod_{k=1}^{q}\binom{b}{2i_{k}+1}\prod_{k=1}^{q}\binom{b}{2j_{k}+1}.
\end{align}
Since the summand is symmetric in $g_{1}$ and $g_{2}$, we obtain
\begin{align}
\lvert N&(b, q, b)\rvert \\
&\ge 2\cdot\f{n}{2^{2g}b^{2q}(q!)^{2}}\sum_{\substack{(g_{1}, g_{2})\\=(0, g)}}(q+2g_{1}-1)! (q+2g_{2}-1)!
\sum_{\substack{(i_{1},\dotsc,i_{q})\vDash g_{1}\\(j_{1},\dotsc,j_{q})\vDash g_{2}}}
\prod_{k=1}^{q}\binom{b}{2i_{k}+1}\prod_{k=1}^{q}\binom{b}{2j_{k}+1} \\
&= \f{qb}{2^{2g-1}b^{2q}(q!)^{2}}(q-1)! (q+2g-1)! \sum_{\substack{(i_{1},\dotsc,i_{q})\vDash 0\\(j_{1},\dotsc,j_{q})\vDash g}}
\prod_{k=1}^{q}\binom{b}{2i_{k}+1}\prod_{k=1}^{q}\binom{b}{2j_{k}+1} \\
&= \f{q!}{2^{2g-1}b^{2q-1}(q!)^{2}}(q+2g-1)! \sum_{(j_{1},\dotsc,j_{q})\vDash g}
\prod_{k=1}^{q}\binom{b}{1}\prod_{k=1}^{q}\binom{b}{2j_{k}+1} \\
&= \f{1}{2^{2g-1}b^{2q-1}q!}(q+q(b-2)+1-1)! \sum_{(j_{1},\dotsc,j_{q})\vDash g}
b^{q}\prod_{k=1}^{q}\binom{b}{2j_{k}+1} \\
\label{eq:Nbqb}
&= \f{(n-q)!}{2^{q(b-2)}b^{q-1}q!}\sum_{(j_{1},\dotsc,j_{q})\vDash g}\prod_{k=1}^{q}\binom{b}{2j_{k}+1}.
\end{align}
The product
$\prod_{k=1}^{q}\binom{b}{2j_{k}+1}$
vanishes whenever $2j_{k}+1>b$ for some $k$.
Hence the sum may be restricted to compositions
$(j_{1},\dotsc,j_{q})$ satisfying
$j_{k} \le (b-1)/2$ ($1\le k \le q$).

Since
\begin{align}
&\f{b-1}{2}\cdot q - g = \f{q(b-1)}{2} - \f{q(b-2)+1}{2} = \f{q-1}{2},
\end{align}
and
$q - (q-1)/2 = (q+1)/2$,
the composition $(j_{1},\dotsc,j_{q})$ consisting of $(q+1)/2$ parts equal to $(b-1)/2$ and
$(q-1)/2$ parts equal to $(b-3)/2$ satisfies the required condition.

For each such composition, we have
\begin{align}
\prod_{k=1}^{q}\binom{b}{2j_{k}+1} = \binom{b}{b}^{\f{q+1}{2}}\binom{b}{b-2}^{\f{q-1}{2}} = \binom{b}{2}^{\f{q-1}{2}}.
\end{align}
Moreover, the number of such compositions is
$\binom{q}{(q-1)/2}$.
Therefore, restricting \eqref{eq:Nbqb} to these compositions, we obtain
\begin{align}
\vt{N(b, q, b)} &\ge \f{(n-q)!}{2^{q(b-2)}b^{q-1}q!}\binom{q}{\f{q-1}{2}}\binom{b}{2}^{\f{q-1}{2}} \\
&= \f{(n-q)!}{2^{q(b-2)}b^{q-1}q!}\cdot \f{q!}{\left(\f{q+1}{2}\right)!\left(\f{q-1}{2}\right)!}\left(\f{b(b-1)}{2}\right)^{\f{q-1}{2}} \\
&= \f{(n-q)!}{2^{q(b-2)}\left(\f{q+1}{2}\right)!\left(\f{q-1}{2}\right)!}\left(\f{b-1}{2b}\right)^{\f{q-1}{2}}.
\end{align}
\end{proof}

\subsection{Upper Bound for \textit{D}}

The following is standard.
\begin{prop}
\label{prop:Zg}
Let $g \in S_{n}$ and suppose that $g$ has cycle type
\begin{align}
\la = (1^{m_{1}}, 2^{m_{2}}, \ldots, n^{m_{n}}).
\end{align}
Define the centralizer of $g$ in $S_{n}$:
\begin{align}
Z_{g} \ceq \{ c \in S_{n} \mid cg = gc \}.
\end{align}
Then $\vt{Z_{g}}$ depends only on the cycle type $\la$, and
\begin{align}
\vt{Z_{g}} = 1^{m_{1}}m_{1}!\, 2^{m_{2}}m_{2}! \cdots n^{m_{n}}m_{n}!.
\end{align}
\end{prop}

\begin{proof}
See \cite[Proposition~1.1.1]{Sagan01}.
\end{proof}

Using this proposition, we obtain the following.

\begin{prop}
\label{prop:Dk}
Let $n \in \Zp$ with $n \ge 2$, $x \in S_{n}$ be an $n$-cycle, and for each $1 \le k \le n-1$, define
\begin{align}
\label{eq:defDk}
D_k \ceq \{\,c\in S_n \mid cx^k=x^k c\,\},
\end{align}
the centralizer of $x^k$ in $S_n$.
Then the following hold.
\begin{enumerate}
\item \label{itm:xcom1}
$D_1=\{x^k \mid 0\le k\le n-1\}$.
\item \label{itm:xcom2}
For any $1\le k,l\le n-1$, if $k\mid l$, then $D_k\subset D_l$.
\item \label{itm:xcom3}
For any $1\le k\le n-1$, let $m=\gcd(n,k)$. Then
\begin{align}
D_k&=D_m, \\
\lvert D_k\rvert&=\lvert D_m\rvert=\left(\f{n}{m}\right)^m m!.
\end{align}
\end{enumerate}
\end{prop}

\begin{proof}
\item[\ref{itm:xcom1}] Since $x$ has cycle type
\begin{align}
(1^{0}, 2^{0}, \ldots, (n-1)^{0}, n^{1}),
\end{align}
\propref{prop:Zg} implies that
\begin{align}
\vt{D_{1}} = n^{1}\cdot 1! = n.
\end{align}
On the other hand,
\begin{align}
 x^{0} = \id,\ x^{1},\ x^{2},\ \ldots,\ x^{n-1}
\end{align}
are all distinct, and for every $0 \le k \le n-1$,
\begin{align}
x^{k} x = x x^{k}.
\end{align}
Hence
\begin{align}
\{ x^{k} \mid 0 \le k \le n-1 \} \subset D_{1},
\end{align}
and the left-hand side contains $n$ elements. Therefore, we conclude that
\begin{align}
D_{1} = \{ x^{k} \mid 0 \le k \le n-1 \}.
\end{align}

\item[\ref{itm:xcom2}] Suppose that $k\mid l$.
Then $l = dk$ for some $d \in \Zp$.
If $c \in D_{k}$, then $c x^{k} = x^{k} c$, and hence
\begin{align}
c x^{l} &= c (x^{k})^{d} = (x^{k})^{d}c = x^{l}c.
\end{align}
Hence $c\in D_l$.

Since $c$ was arbitrary, we obtain $D_{k} \subset D_{l}$.
\HL

\item[\ref{itm:xcom3}] Since $m\mid k$, \ref{itm:xcom2} implies that
\begin{align}
\label{eq:Dmk}
D_{m} \subset D_{k}.
\end{align}

We can write $n = mn'$ and $k = mk'$ with $n', k' \in \Zp$ and $\gcd(n', k') = 1$. Hence there exists
$r, s \in \Z$ such that
\begin{align}
rn' + sk' = 1.
\end{align}
Then
\begin{align}
rn + sk = rmn' + smk' = m(rn'+sk') = m.
\end{align}
Since $x^n=1$, we have
\begin{align}
x^{m} = x^{rn+sk}=x^{sk}=(x^{k})^{s}.
\end{align}
Let $c \in D_{k}$. Then
\begin{align}
cx^{m} = c(x^{k})^{s} = (x^{k})^{s}c = x^{m}c,
\end{align}
hence $c \in D_{m}$. Since $c$ was arbitrary, we obtain
\begin{align}
\label{eq:Dkm}
D_{k} \subset D_{m}.
\end{align}

By \eqref{eq:Dmk} and \eqref{eq:Dkm}, we conclude that $D_{k} = D_{m}$.

Since $m \mid n$, the cycle type of $x^{m}$ is $((n/m)^{m})$. Therefore, by \propref{prop:Zg},
\begin{align}
\vt{D_{k}} = \vt{D_{m}} = \left(\f{n}{m}\right)^{m} m!.
\end{align}
\end{proof}

\begin{cor}
\label{cor:ADtrivial}
Let $x, y \in S_{n}$, and assume $x$ is an $n$-cycle. Let $\msD$ be the dessin corresponding to the monodromy
group $\gen{x, y}$. Then,
\begin{enumerate}
\item \label{itm:ADtrivial1} $\AD = \{ x^{k} \mid 0 \le k \le n-1,\ x^{k}y = yx^{k} \}$.
\item \label{itm:ADtrivial2} $\AD$ is trivial if and only if $x^{k}y \ne yx^{k}$ for all $1 \le k \le n-1$.
\end{enumerate}
\end{cor}

\begin{proof}
\item[\ref{itm:ADtrivial1}] $\AD$ is defined by \eqref{eq:cxcy}. By \propref{prop:Dk}\ref{itm:xcom1}, the elements commuting with $x$
are precisely the powers of $x$. Hence
\begin{align}
\AD = \{ x^{k} \mid 0 \le k \le n-1,\ x^{k}y = yx^{k} \}.
\end{align}

\item[\ref{itm:ADtrivial2}] By \ref{itm:ADtrivial1}, we have
\begin{align}
\AD\ \text{is trivial} \iff \AD = \{ x^{0} \} \iff x^{k}y \ne yx^{k}\ \text{for all}\ 1 \le k \le n-1.
\end{align}
\end{proof}

\begin{prop}
\label{prop:D}
Let $n \in \Zp$ with $n \ge 2$, $x \in S_{n}$ be an $n$-cycle, and for each $1 \le k \le n-1$, define $D_{k}$ by \eqref{eq:defDk}.
Let $D$ be the set of permutations $y \in S_{n}$ such that the corresponding dessin $\msD$,
with monodromy group $\gen{x,y}$, has nontrivial automorphism group. Then
\begin{align}
D &= \bigcup_{\ell\colon \text{prime},\, \ell\, \mid\, n} D_{\f{n}{\ell}}, \\
\label{eq:vtD}
\vt{D} &\le \sum_{\ell\colon \text{prime},\, \ell\, \mid\, n} \ell^{\f{n}{\ell}}\left(\f{n}{\ell}\right)!.
\end{align}
\end{prop}

\begin{proof}
By \corref{cor:ADtrivial}\ref{itm:ADtrivial2}, we have
\begin{align}
D^{c} = \bigcap_{k=1}^{n-1} D_{k}^{c},
\end{align}
hence
\begin{align}
D = \bigcup_{k=1}^{n-1} D_{k}.
\end{align}

By \propref{prop:Dk}\ref{itm:xcom2}\ref{itm:xcom3}, every $D_k$ is contained in
$D_m$ for some maximal proper divisor $m$ of $n$. Therefore, to determine $D$
and estimate its cardinality, it suffices to consider only those $k$ of the form
$n/\ell$, where $\ell$ is a prime divisor of $n$.
Hence
\begin{align}
D = \bigcup_{\ell\colon \text{prime},\, \ell\, \mid\, n} D_{\f{n}{\ell}}.
\end{align}
Since $\gcd(n, n/\ell) = n/\ell$, \propref{prop:Dk}\ref{itm:xcom3} yields
\begin{align}
\label{eq:Dnp}
\vt{D_{\f{n}{\ell}}} = \ell^{\f{n}{\ell}}\left(\f{n}{\ell}\right)!.
\end{align}
Consequently,
\begin{align}
\vt{D} \le \sum_{\ell\colon \text{prime},\, \ell\, \mid\, n} \vt{D_{\f{n}{\ell}}}
=  \sum_{\ell\colon \text{prime},\, \ell\, \mid\, n} \ell^{\f{n}{\ell}}\left(\f{n}{\ell}\right)!.
\end{align}
\end{proof}

\begin{lemma}
\label{lem:Dupper}
Let $n \in \Zp$ with $n \ge 2$, $x \in S_{n}$ be an $n$-cycle, and for each $1 \le k \le n-1$, define $D_{k}$ by \eqref{eq:defDk}.
Let
\begin{align}
\kappa_{1} \ceq \f{2623}{1894} = 1.384\ldots, \q \kappa_{2} \ceq \f{972}{947} = 1.026\ldots.
\end{align}
Then
\begin{align}
\label{eq:Dupper}
\vt{D} \le \kappa_{1}\cdot 2^{\f{n}{2}}\left(\f{n}{2}\right)!,
\end{align}
where, for non-integer $M$, we define $M! \ceq \Gamma(M+1)$.

In particular, if $n$ is odd, then
\begin{align}
\label{eq:Dupperodd}
\vt{D} \le \kappa_{2}\cdot3^{\f{n}{3}}\left(\f{n}{3}\right)!.
\end{align}
\end{lemma}

\begin{proof}
When $n = 2$, by \propref{prop:D},
\begin{align}
\vt{D} = \vt{D_{1}} = 2^{\f{n}{2}}\left(\f{n}{2}\right)! \le \kappa_{1}\cdot 2^{\f{n}{2}}\left(\f{n}{2}\right)! .
\end{align}

In the following, assume $n \ge 3$.

Let
\begin{align}
\label{eq:Fnx}
F_{n}(x) \ceq x^{\f{n}{x}}\Gamma\left(\f{n}{x}+1\right)\q (x \ge 2).
\end{align}
Then, for every prime divisor $\ell$ of $n$, by \propref{prop:Dk}\ref{itm:xcom3},
\begin{align}
\vt{D_{\f{n}{\ell}}} = F_{n}(\ell).
\end{align}

Let $G_{n}(x) = \log F_{n}(x)$. Then by \eqref{eq:Fnx},
\begin{align}
G_{n}(x) = \f{n}{x}\log x + \log \Gamma\left(\f{n}{x}+1\right),
\end{align}
hence
\begin{align}
G_{n}'(x) = -\f{n}{x^{2}}\left(\log x - 1 + \psi\left(\f{n}{x}+1\right)\right),
\end{align}
where
\begin{align}
\psi(x) \ceq \dv{x} \log \Gamma(x) = \f{\Gamma'(x)}{\Gamma(x)}
\end{align}
is the digamma function.

It is well known that
\begin{align}
\psi(x+1) > \log x \q (x>0).
\end{align}
Hence
\begin{align}
G_{n}'(x) &< -\f{n}{x^{2}}\left(\log x - 1 + \log \f{n}{x}\right) \\
&= -\f{n}{x^{2}}(\log n - 1) \le -\f{n}{x^{2}}(\log 3 - 1) < 0.
\end{align}
Therefore, $G_n$ is strictly decreasing in $x$. Hence so is $F_n$.
\HL

\item[(i)] The case where $n$ has precisely one prime divisor.

Let $\ell$ denote the unique prime divisor of $n$. Then we can write $n = \ell^{m}$ with
$m \in \Zp$. By \propref{prop:D}, we have
\begin{align}
D &= D_{\f{n}{\ell}}, \\
\vt{D} &= \ell^{\f{n}{\ell}}\left(\f{n}{\ell}\right)! = F_{n}(\ell).
\end{align}

Since $F_n$ is decreasing and $\ell \ge 2$, we have
\begin{align}
\vt{D} = F_{n}(\ell) \le F_{n}(2) = 2^{\f{n}{2}}\left(\f{n}{2}\right)! \le  \kappa_{1}\cdot2^{\f{n}{2}}\left(\f{n}{2}\right)!.
\end{align}

In particular, if $n$ is odd, then $\ell \ge 3$, and therefore
\begin{align}
\vt{D} = F_{n}(\ell) \le F_{n}(3) = 3^{\f{n}{3}}\left(\f{n}{3}\right)! \le \kappa_{2}\cdot3^{\f{n}{3}}\left(\f{n}{3}\right)!.
\end{align}

\item[(ii)] The case where $n$ has at least~2 prime divisors.

Let
\begin{align}
\ell_1<\ell_2<\cdots<\ell_w \q (w \ge 2)
\end{align}
be the distinct prime divisors of $n$.
For $1 \le i \le w-1$, write $n = \ell_{i}\ell_{i+1}s_{i}$ with $s_{i} \in \Zp$.
Since $F_{n}(\ell)=\vt{D_{n/\ell}}$, we have
\begin{align}
\f{F_{n}(\ell_{i+1})}{F_{n}(\ell_{i})} &= \f{\vt{D_{\f{n}{\ell_{i+1}}}}}{\vt{D_{\f{n}{\ell_{i}}}}}
= \f{\ell_{i+1}^{\f{n}{\ell_{i+1}}}\left(\f{n}{\ell_{i+1}}\right)!}{\ell_{i}^{\f{n}{\ell_{i}}}\left(\f{n}{\ell_{i}}\right)!}
= \f{\ell_{i+1}^{\ell_{i}s_{i}}(\ell_{i}s_{i})!}{\ell_{i}^{\ell_{i+1}s_{i}}(\ell_{i+1}s_{i})!}.
\end{align}

Define
\begin{align}
H(m, d, s) \ceq \f{(m+d)^{ms}(ms)!}{m^{(m+d)s}((m+d)s)!}\q (m, d, s \in \Zp, m \ge 2),
\end{align}
so that
\begin{align}
\f{F_{n}(\ell_{i+1})}{F_{n}(\ell_{i})} &= \f{\vt{D_{\f{n}{\ell_{i+1}}}}}{\vt{D_{\f{n}{\ell_{i}}}}} = H(\ell_{i}, \ell_{i+1}-\ell_{i}, s_{i}).
\end{align}

Then we have
\begin{align}
\f{H(m, d, s)}{H(m, d, 1)^{s}} &= \f{(m+d)^{ms}(ms)!}{m^{(m+d)s}((m+d)s)!}\cdot\left(\f{m^{(m+d)}(m+d)!}{(m+d)^{m}m!}\right)^{s} \\
&= \f{(ms)!}{((m+d)s)!}\left(\f{(m+d)!}{m!}\right)^{s} \\
&= \f{(m+1)^{s}(m+2)^{s}\cdots(m+d)^{s}}{(ms+1)(ms+2)\cdots(ms+ds)} \\
&= \f{(m+1)^{s}}{(ms+1)\cdots(ms+s)}\cdot\f{(m+2)^{s}}{(ms+s+1)\cdots(ms+2s)}\cdot \\
&\qquad \cdots\cdot\f{(m+d)^{s}}{(ms+(d-1)s+1)\cdots(ms+ds)} \\
&= \prod_{i=1}^{d}\f{(m+i)^{s}}{\prod_{j=1}^{s}(ms+(i-1)s+j)}.
\end{align}
For each $1\le i\le d$ and $1\le j\le s$,
\begin{align}
ms+(i-1)s+j = s(m+i-1)+j \ge (m+i-1)+1 = m+i.
\end{align}
Hence
\begin{align}
\label{eq:Hmds1}
H(m, d, s) \le H(m, d, 1)^{s}.
\end{align}

Since $(1+d/m)^{m}$ is strictly increasing as a function of $m > 0$
and converges to $e^{d}$ as $m \to \infty$,
\begin{align}
\left(\f{m+d}{m}\right)^{m} = \left(1+\f{d}{m}\right)^{m} < e^{d} < m^{d} \q (m>e).
\end{align}
Hence
\begin{align}
(m+d)^{m} < m^{m+d} \q (m > e).
\end{align}
Therefore,
\begin{align}
\label{eq:Hmge3}
H(m, d, 1) = \f{(m+d)^{m}}{m^{m+d}}\cdot\f{1}{(m+1)\cdots(m+d)} < 1\q (m \ge 3).
\end{align}
Moreover, when $m = 2$,
\begin{align}
H(2, 1, 1) &= \f{3^{2}\cdot 2!}{2^{3}\cdot 3!} = \f{3}{8} < 1, \\
\f{H(2, d+1, 1)}{H(2, d, 1)} &= \f{(d+3)^{2}\cdot 2!}{2^{d+3}(d+3)!}\cdot\f{2^{d+2}(d+2)!}{(d+2)^{2}\cdot 2!} \\
&= \f{d+3}{2(d+2)^{2}} < \f{d+3}{2(d+2)} < \f{d+3}{d+4} < 1\q(d \ge 1).
\end{align}
Hence $H(2,d,1)<1$ for all $d \ge 1$.
Combining this with \eqref{eq:Hmge3}, we obtain
\begin{align}
H(m, d, 1) < 1 \q (m \ge 2,\ d \ge 1).
\end{align}
By \eqref{eq:Hmds1},
\begin{align}
\label{eq:Hmds2}
H(m, d, s) \le H(m, d, 1) < 1 \q (m \ge 2,\ d \ge 1,\ s \ge 1).
\end{align}

Moreover, regarding $H(m, d, 1)$,
\begin{align}
\f{H(m, d+1, 1)}{H(m, d, 1)} &= \f{(m+d+1)^{m}m!}{m^{m+d+1}(m+d+1)!}\cdot\f{m^{m+d}(m+d)!}{(m+d)^{m}m!} \\
&= \f{(m+d+1)^{m-1}}{m(m+d)^{m}} = \f{1}{m(m+d)}\left(1+\f{1}{m+d}\right)^{m-1} \\
&< \f{1}{m(m+d)}\left(1+\f{1}{m+d}\right)^{m+d}.
\end{align}
Since
\begin{align}
\left(1+\f{1}{m+d}\right)^{m+d} < e,
\end{align}
we obtain
\begin{align}
\f{H(m, d+1, 1)}{H(m, d, 1)} &< \f{e}{m(m+d)} \le \f{e}{2(2+1)} < 1.
\end{align}
Therefore, $H(m,d,1)$ is strictly decreasing in $d$.

Moreover,
\begin{align}
\f{H(m+1, d, 1)}{H(m, d, 1)} &= \f{(m+d+1)^{m+1}(m+1)!}{(m+1)^{m+d+1}(m+d+1)!} \cdot\f{m^{m+d}(m+d)!}{(m+d)^{m}m!} \\
&= \left(1 + \f{1}{m+d}\right)^{m}\left(\f{m}{m+1}\right)^{m+d} \\
&< \left(1 + \f{1}{m}\right)^{m}\left(\f{m}{m+1}\right)^{m+d} = \left(\f{m}{m+1}\right)^{d} < 1.
\end{align}
Therefore, $H(m, d, 1)$ is also strictly decreasing in $m$.

Since each $\ell_{i}$ is a prime number, we have $\ell_{i+1} - \ell_{i} = 1$ only when $i = 1$, $\ell_{1} = 2$, and
$\ell_{2} = 3$. Otherwise $\ell_{i+1} - \ell_{i} \ge 2$. Hence,
\begin{align}
\f{F_{n}(\ell_{2})}{F_{n}(\ell_{1})} &= H\left(\ell_{1}, \ell_{2}-\ell_{1}, \f{n}{\ell_{1}\ell_{2}}\right) \\
&\le H(2, 1, 1) = \f{3^{2}\cdot 2!}{2^{3}\cdot 3!} = \f{3}{8},
\end{align}
and for $i \ge 2$,
\begin{align}
\f{F_{n}(\ell_{i+1})}{F_{n}(\ell_{i})} &= H\left(\ell_{i}, \ell_{i+1}-\ell_{i}, \f{n}{\ell_{i}\ell_{i+1}}\right) \\
&\le H(3, 2, 1) = \f{5^{3}\cdot 3!}{3^{5}\cdot 5!} = \f{25}{972}.
\end{align}
Therefore,
\begin{align}
F_{n}(\ell_{i}) \le
\begin{dcases}
F_{n}(2) & (i = 1) \\
\f{3}{8}\left(\f{25}{972}\right)^{i-2}F_{n}(2) & (i \ge 2)
\end{dcases}.
\end{align}
Thus,
\begin{align}
\vt{D} &\le \sum_{i=1}^{w}F_{\ell_{i}} \le F_{n}(2) + \sum_{i=2}^{w}\f{3}{8}\left(\f{25}{972}\right)^{i-2}F_{n}(2) \\
&\le F_{n}(2) + \f{3}{8}\sum_{i=2}^{\infty}\left(\f{25}{972}\right)^{i-2}F_{n}(2) \\
&= \left(1 + \f{3}{8}\cdot\f{1}{1-\f{25}{972}}\right)F_{n}(2) = \f{2623}{1894}F_{n}(2) \\
&= \kappa_{1}\cdot 2^{\f{n}{2}}\left(\f{n}{2}\right)!.
\end{align}

When $n$ is odd, we have $\ell_{1} \ge 3$ and $\ell_{i+1} - \ell_{i} \ge 2$ for all $i$. Hence,
\begin{align}
\f{F_{n}(\ell_{i+1})}{F_{n}(\ell_{i})} &\le H(3, 2, 1) =  \f{25}{972}.
\end{align}
Therefore, for all $1 \le i \le w$,
\begin{align}
F_{n}(\ell_{i}) &\le \left(\f{25}{972}\right)^{i-1}F_{n}(3).
\end{align}
Thus,
\begin{align}
\vt{D} &\le \sum_{i=1}^{w}F_{n}(\ell_{i}) \le \sum_{i=1}^{w}\left(\f{25}{972}\right)^{i-1}F_{n}(3) \\
&\le \sum_{i=1}^{\infty}\left(\f{25}{972}\right)^{i-1}F_{n}(3) \\
&= \f{1}{1-\f{25}{972}}F_{n}(3) = \f{972}{947}F_{n}(3) \\
&= \kappa_{2}\cdot 3^{\f{n}{3}}\left(\f{n}{3}\right)!.
\end{align}

Combining (i) and (ii), we obtain
\begin{align}
\vt{D} \le \kappa_{1}\cdot 2^{\f{n}{2}}\left(\f{n}{2}\right)!,
\end{align}
and, if $n$ is odd,
\begin{align}
\vt{D} \le \kappa_{2}\cdot 3^{\f{n}{3}}\left(\f{n}{3}\right)!.
\end{align}
\end{proof}

\section{Passport $[b^{q}, b^{q}, n]$ with Genus $\ge 2$}
\label{sec:class3}

In this section, we consider the passports of the form $[b^{q}, b^{q}, n]$,
where $n = qb$, $q \ge 2$, and the genus is at least~2.
A dessin with such a passport has a single face, equal numbers of black and white vertices, and all vertices have the same valency.

By \eqref{eq:genus}, the genus of such a dessin is
\begin{align}
\label{eq:genusclass3}
g = \f{n - (2q+1)}{2} + 1 = \f{q(b-2)+1}{2} \ge 2.
\end{align}
Since $g$ is an integer, both $b$ and $q$ must be odd. Consequently, $n$ is also odd, and in particular $b, q \ge 3$.

\figref{fig:dessins-class3} shows examples of \ddes with passport $[3^{3}, 3^{3}, 9]$. The dessin on the left has
automorphism group $C_{3}$, while the dessin on the right has trivial automorphism
group.

\begin{figure}[htbp]
\centering
\includegraphics[width=150mm]{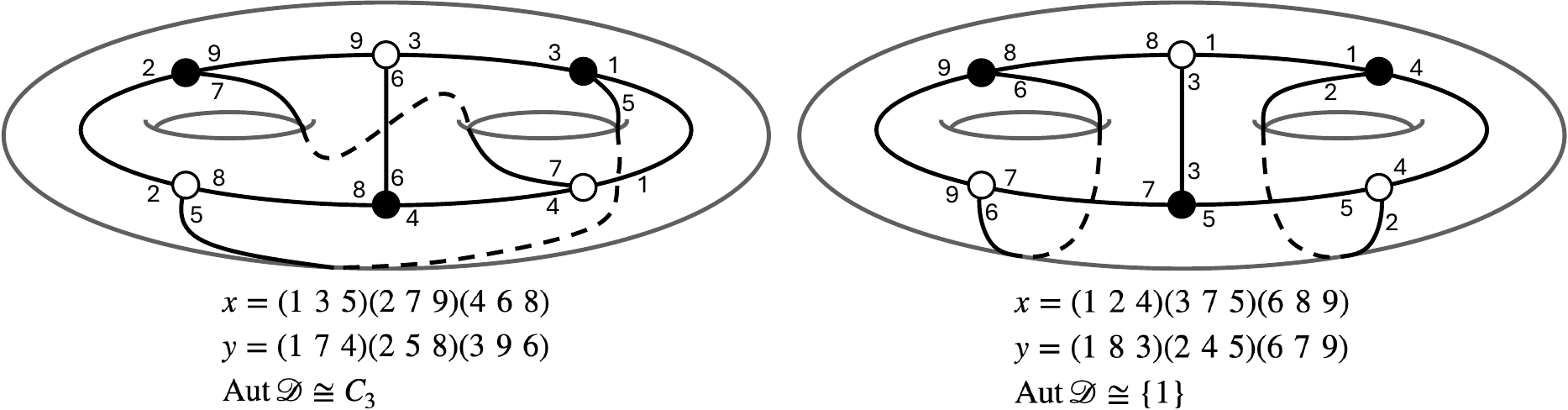} 
\caption{\ddes with passport $[3^{3}, 3^{3}, 9]$}
\label{fig:dessins-class3}
\end{figure}

By \cite[Theorem~6.1]{Ohnishi26}, the passport $[a^{p}, b^{q}, n]$ has a regular dessin if and only if
$\gcd(p, q) = 1$. Since $\gcd(q, q)=q \ge 2$, the passport $[b^{q}, b^{q}, n]$ never admits a regular dessin.

We now prove that every uniform passport $[b^{q}, b^{q}, n]$ ($q \ge 2$)
of genus at least~2 admits a dessin with trivial automorphism group.
The proof is divided into two cases, according as $b=3$ or $b\ge5$.

\subsection{The Subcase $b = 3$}

We construct permutations $x$ and $y$ for each passport
$[3^{q}, 3^{q}, 3q]$ of genus at least~2
such that the corresponding dessin has a trivial automorphism group.
By \eqref{eq:genusclass3}, $q$ is odd and $q \ge 3$.

Since the automorphism group is invariant under permutations of the three partitions,
it suffices to consider the passport
$[3q,3^{q},3^{q}]$,
that is, the case where $x$, $y$, and
$z=(xy)^{-1}$ have cycle types
$(3q)$, $(3^{q})$, and $(3^{q})$, respectively.
Since the automorphism group is also invariant under relabeling of the edges,
we may fix $x$.

\begin{prop}
\label{prop:c3b3}
For each odd integer $q \ge 3$, let $x = (1\ 2\ \ldots\ 3q) \in S_{3q}$ and define $y \in S_{3q}$ for each $q$ as follows:
\begin{align}
y &= \A \B_{0}\B_{1}\cdots\B_{\f{q-5}{2}} \G, \\
\A &= (1\ 2\ 4)(3\ 7\ 5), \\
\B_{i} &=
\begin{dcases}
(11v{+}6\ 11v{+}9\ 11v{+}10)(11v{+}8\ 3q{-}v\ 11v{+}13) & (i \equiv 0 \npmod{2}) \\
(11v\ 11v{+}1\ 11v{+}4)(11v{+}3\ 11v{+}7\ 11v{+}5) & (i \equiv 1 \npmod{2})
\end{dcases}, \\
&\text{where}\ v = \fl{\f{i}{2}}, \\
\G &=
\begin{dcases}
(11w{+}11\ 11w{+}12\ 11w{+}14) & (q \equiv 1 \npmod{4}) \\
(11w{+}6\ 11w{+}8\ 11w{+}9) & (q \equiv 3 \npmod{4})
\end{dcases}, \\
\label{eq:yqdef}
&\text{where}\ w = \fl{\f{q-3}{4}}.
\end{align}
Then the following hold:
\begin{enumerate}
\item\label{itm:c3b3-1} $y$ has cycle type $(3^{q})$.
\item\label{itm:c3b3-2} $(xy)^{-1}$ has cycle type $(3^{q})$.
\item\label{itm:c3b3-3} For the \dde $\msD$ corresponding to $x$ and $y$, we have $\AD \cong \{ 1 \}$.
\end{enumerate}
\end{prop}

\begin{rem}
For $q = 3, 5, 7,$ and $9$, the permutations $y$ and $(xy)^{-1}$ are as follows:
\begin{itemize}
\item $q=3$:
\begin{align}
y&=(1\ 2\ 4)(3\ 7\ 5)(6\ 8\ 9), \\
(xy)^{-1}&=(1\ 8\ 3)(2\ 4\ 5)(6\ 7\ 9)
\end{align}
This pair $(x,y)$ corresponds to the dessin on the right in \figref{fig:dessins-class3} after permuting the roles of
$y \mapsto x$, and $(xy)^{-1}\mapsto y$.
\item $q=5$:
\begin{align}
y&=(1\ 2\ 4)(3\ 7\ 5)(6\ 9\ 10)(8\ 15\ 13)(11\ 12\ 14), \\
(xy)^{-1}&=(1\ 8\ 3)(2\ 4\ 5)(6\ 7\ 10)(9\ 13\ 11)(12\ 14\ 15)
\end{align}
\item $q=7$:
\begin{align}
y&=(1\ 2\ 4)(3\ 7\ 5)(6\ 9\ 10)(8\ 21\ 13)(11\ 12\ 15)(14\ 18\ 16)(17\ 19\ 20), \\
(xy)^{-1}&=(1\ 8\ 3)(2\ 4\ 5)(6\ 7\ 10)(9\ 13\ 11)(12\ 15\ 16)(14\ 21\ 19)(17\ 18\ 20)
\end{align}
\item $q=9$:
\begin{align}
y &= (1\ 2\ 4)(3\ 7\ 5)(6\ 9\ 10)(8\ 27\ 13)(11\ 12\ 15)(14\ 18\ 16)(17\ 20\ 21) \\
&\q\ (19\ 26\ 24)(22\ 23\ 25), \\
(xy)^{-1}&= (1\ 8\ 3)(2\ 4\ 5)(6\ 7\ 10)(9\ 13\ 11)(12\ 15\ 16)(14\ 27\ 19)(17\ 18\ 21) \\
&\q\ (20\ 24\ 22)(23\ 25\ 26)
\end{align}
\end{itemize}
These examples illustrate the general pattern of the construction.
\end{rem}

\begin{proof}
\ref{itm:c3b3-1}
In \eqref{eq:yqdef}, every cycle in $y$ has length~$3$, and the number of these cycles is
\begin{align}
2 + 2\cdot\left(\f{q-5}{2}+1\right) + 1 = q.
\end{align}
Therefore, to prove that $y$ has cycle type $(3^{q})$, it suffices to show that every element of $\{ 1, \ldots, 3q \}$ appears
in \eqref{eq:yqdef}.
\HL

\item[(i)] Case $q \equiv 1 \pmod{4}$.

We have
\begin{align}
w &=  \fl{\f{q-3}{4}} = \f{q-5}{4},
\end{align}
and hence
\begin{align}
3q = 12w+15.
\end{align}
Moreover, the index $i$ of $\B_{i}$ ranges over
\begin{align}
&0 \le i \le \f{q-5}{2} =  2w.
\end{align}

We now verify that every element $s \in \{1, \ldots, 3q\} = \{1, \ldots, 12w+15\}$ appears in \eqref{eq:yqdef}.
\begin{itemize}
\item If $s \in \{ 1, 2, 3, 4, 5, 7 \}$, then $s$ appears in $\A$.
\item If $s \in \{ 6, 8, 9, 10 \}$, then $s$ appears in $\B_{0}$.
\item If $11 \le s \le 11w+10$, write $s = 11t+u$ with $0 \le u \le 10$. Then $1 \le t \le w$,
and $u$ ranges from $0$ to $10$ for each $t$. In this case, $s$ appears in
\begin{align}
\begin{dcases}
\B_{2t-2} & (u = 2) \\
\B_{2t-1} & (u \in \{ 0, 1, 3, 4, 5, 7 \}) \\
\B_{2t} & (u \in \{ 6, 8, 9, 10 \})
\end{dcases}.
\end{align}
\item The elements $11w+11$, $11w+12$, and $11w+14$ appear in $\G$.
\item The element $11w+13$ appears in $\B_{2w} = \B_{\f{q-5}{2}}$.
\item If $11w+15 \le s \le 12w+15$, then $s$ appears in $\B_{2(12w+15-s)}$ as $3q-(12w+15-s)$.
\end{itemize}
Therefore, every element of $\{ 1, \ldots, 3q \}$ appears in  \eqref{eq:yqdef}.
\HL

\item[(ii)] Case $q \equiv 3 \pmod{4}$.

We have
\begin{align}
w &=  \fl{\f{q-3}{4}} = \f{q-3}{4},
\end{align}
and hence
\begin{align}
3q = 12w+9.
\end{align}
Moreover, the index $i$ of $\B_{i}$ ranges over
\begin{align}
0 \le i \le \f{q-5}{2} =  2w-1 \q (q \ge 7).
\end{align}

We now verify that every element $s \in \{1, \ldots, 3q\} = \{1, \ldots, 12w+9\}$ appears in \eqref{eq:yqdef}.
\begin{itemize}
\item If $s \in \{ 1, 2, 3, 4, 5, 7 \}$, then $s$ appears in $\A$.
\item The elements $11w+6$, $11w+8$, and $11w+9$ appear in $\G$.
\end{itemize}
If $q \ge 7$, then
\begin{itemize}
\item If $s \in \{ 6, 8, 9, 10 \}$, then $s$ appears in $\B_{0}$.
\item If $11 \le s \le 11w+5$, write $s = 11t+u$ with $0 \le u \le 10$. Then $1 \le t \le w$,
and $u$ ranges from $0$ to $10$ when $1 \le t \le w-1$ and $0$ to $5$ when $t = w$.
In this case, $s$ appears in
\begin{align}
\begin{dcases}
\B_{2t-2} & (u = 2) \\
\B_{2t-1} & (u \in \{ 0, 1, 3, 4, 5, 7 \}) \\
\B_{2t} & (u \in \{ 6, 8, 9, 10 \})
\end{dcases}.
\end{align}
\item The element $11w+7$ appears in $\B_{2w-1} = \B_{\f{q-5}{2}}$.
\item If $11w+10 \le s \le 12w+9$, then $s$ appears in $\B_{2(12w+9-s)}$ as $3q-(12w+9-s)$.
\end{itemize}
Therefore, every element of $\{ 1, \ldots, 3q \}$ appears in  \eqref{eq:yqdef}.
\HL

By (i) and (ii), we conclude that $y$ has cycle type $(3^{q})$.
\HL

\item[\ref{itm:c3b3-2}]
Since $xy$ and $(xy)^{-1}$ have the same cycle type, it suffices to show that $xy$ has cycle type $(3^q)$.
We now prove this.

\tabref{tab:xycycles} lists the $q$ cycles of length~3 in $xy$ arising from \eqref{eq:yqdef}.

\begin{table}[htbp]
\centering
\small
\begin{tabular}{|c|c|c|}
\hline
$s$ & $y\cdot s$ & $xy\cdot s$ \\
\hline
1 & 2 & 3 \\
3 & 7 & 8 \\
8 & $3q$ & 1 \\
\hline
2 & 4 & 5 \\
5 & 3 & 4 \\
4 & 1 & 2 \\
\hline
\multicolumn{3}{|c|}{$6 \le s \le 11w+5$ ($s \ne 8$), $s=11w+8$} \\
\hline
$11v+6$ & $11v+9$ & $11v+10$ \\
$11v+10$ & $11v+6$ & $11v+7$ \\
$11v+7$ & $11v+5$ & $11v+6$ \\
\hline
$11v+9$ & $11v+10$ & $11v+11$ \\
$11(v+1)$ & $11(v+1)+1$ & $11v+13$ \\
$11v+13$ & $11v+8$ & $11v+9$ \\
\hline
$11v+1$ & $11v+4$ & $11v+5$ \\
$11v+5$ & $11v+3$ & $11v+4$ \\
$11v+4$ & $11v$ & $11v+1$ \\
\hline
$11v+3$ & $11v+7$ & $11v+8$ \\
$11v+8$ & $3q-v$ & $3q-(v-1)$ \\
$3q-(v-1)$ & $11(v-1)+13$ & $11v+3$ \\
\hline
\multicolumn{3}{|c|}{Case $q \equiv 1 \pmod{4}$} \\
\hline
$11w+6$ & $11w+9$ & $11w+10$ \\
$11w+10$ & $11w+6$ & $11w+7$ \\
$11w+7$ & $11w+5$ & $11w+6$ \\
\hline
$11w+9$ & $11w+10$ & $11w+11$ \\
$11w+11$ & $11w+12$ & $11w+13$ \\
$11w+13$ & $11w+8$ & $11w+9$ \\
\hline
$11w+12$ & $11w+14$ & \makecell{$3q-w$\\($=11w+15$)} \\
$3q-w$ & $11w+13$ & $11w+14$ \\
$11w+14$ & $11w+11$ & $11w+12$ \\
\hline
\multicolumn{3}{|c|}{Case $q \equiv 3 \pmod{4}$}\\
\hline
$11w+6$ & $11w+8$ & $11w+9$ \\
$11w+9$ & $11w+6$ & $11w+7$ \\
$11w+7$ & $11w+5$ & $11w+6$ \\
\hline
\end{tabular}
\vspace{1\baselineskip}
\caption{Cycles of $xy$ ($s = 1, \ldots, 3q$)}\label{tab:xycycles}
\end{table}

We now verify that every element $s \in \{1, \ldots, 3q\}$ appears in one of the cycles of $xy$.
Note that
\begin{align}
3q =
\begin{dcases}
12w + 15 & (q \equiv 1 \pmod{4}) \\
12w + 9 & (q \equiv 3 \pmod{4})
\end{dcases}.
\end{align}
\begin{itemize}
\item If $1 \le s \le 5$ or $s = 8$, then $s$ appears in the first two cycles.
\item If $6 \le s \le 11w+5$ and $s \ne 8$, write $s = 11t+u$ with $0 \le u \le 10$.
Then, if $u = 0$, we have $t \ge 1$ and $s$ appears as $11(v+1)$ with $v = t-1$.
If $u = 2$, we have $t \ge 1$ and $s$ appears as $11v+13$ with $v = t-1$.
Otherwise it appears
as $11v+u$ with $v = t$.
\item The element $s = 11w+8$ appears as $11v+8$ with $v=w$.
\item Suppose that $q \equiv 1 \pmod{4}$:
\begin{itemize}
\item If $11w+6 \le s \le 11w+14$ and $s \ne 11w+8$, then $s$ appears in the last three cycles.
\item If $s = 11w+15$, then $s$ appears as $3q-w$.
\item If $11w+16 \le s \le 12w+15$, then $s$ appears as $3q - v$, where $v = 12w+15 - s$.
\end{itemize}
\item Suppose that $q \equiv 3 \pmod{4}$:
\begin{itemize}
\item If $11w+6 \le s \le 11w+9$ and $s \ne 11w+8$, then $s$ appears in the last cycle.
\item If $11w+10 \le s \le 12w+9$, then $s$ appears as $3q-v$, where $v = 12w+9 - s$.
\end{itemize}
\end{itemize}

Therefore, every element appears in exactly one cycle, and hence $xy$ has cycle type $(3^q)$.
\HL

\item[\ref{itm:c3b3-3}]
By \corref{cor:ADtrivial}\ref{itm:ADtrivial2}, it suffices to show that
\begin{align}
\text{for every}\ 1 \le k \le 3q-1, \ x^{k}y \ne yx^{k}.
\end{align}
Equivalently,
\begin{align}
&\text{for every}\ 1 \le k \le 3q-1,\\
\label{eq:3qxkytrivial}
&\q\text{there exists}\ e \in E = \{ 1, \ldots, 3q \}\ \text{such that}\ x^{k}y\cdot e \ne yx^{k}\cdot e.
\end{align}

In what follows, the elements of $E$ are taken modulo $3q$.
\HL

\item[(i)] Let $e = 1$. Then
\begin{align}
x^{k}y\cdot e &= x^{k} \cdot 2 = k+2, \\
yx^{k}\cdot e &= y\cdot(k+1).
\end{align}
The equality $x^{k}y\cdot e = yx^{k}\cdot e$ holds only if
$k+1$ is mapped to $k+2$ by $y$. By \eqref{eq:yqdef}, this occurs only when
\begin{align}
k + 1 = 1,\ 11v+9,\ 11v+8,\ 3q-v,\ 11v,\ 11w+11,\ \text{or}\ 11w+8,
\end{align}
that is,
\begin{align}
k = 3q,\ 11v+8,\ 11v+7,\ 3q-v-1,\ 11v-1,\ 11w+10,\ \text{or}\ 11w+7.
\end{align}

Since $1 \le k \le 3q-1$, the case $k = 3q$ is excluded. 

Assume that $k = 11v+7$. If $x^{k}y\cdot e = yx^{k}\cdot e$,
then we must have
\begin{align}
3q - v = 11v+8+1,
\end{align}
and hence
\begin{align}
\label{eq:k11v7-1}
v = \f{q-3}{4}.
\end{align}
On the other hand, by \eqref{eq:yqdef}, this case corresponds to
\begin{align}
\label{eq:qicond}
0 \le i \le \f{q-5}{2},\ v = \f{i}{2}.
\end{align}
Therefore,
\begin{align}
\label{eq:k11v7-2}
v \le \f{q-5}{4}.
\end{align}
However, \eqref{eq:k11v7-1} and \eqref{eq:k11v7-2} are incompatible.
Therefore, the case $k = 11v+7$ is excluded.

Assume that $k = 3q-v-1$. If $x^{k}y\cdot e = yx^{k}\cdot e$,
then we must have
\begin{align}
11v+13 = 3q-v+1,
\end{align}
and hence
\begin{align}
\label{eq:k3q-v}
v = \f{q-4}{4}.
\end{align}
On the other hand, this case also corresponds to \eqref{eq:qicond}, and therefore
\eqref{eq:k11v7-2} must hold. Since \eqref{eq:k11v7-2} and \eqref{eq:k3q-v} are incompatible,
the case $k = 3q-v-1$ is excluded.

Thus, for $e = 1$, the equality $x^{k}y\cdot e = yx^{k}\cdot e$ holds only if
\begin{align}
\label{eq:ke1}
k = 11v+8,\ 11v-1,\ 11w+7,\ \text{or}\ 11w+10.
\end{align}

\item[(ii)] Let $e = 3$. Then
\begin{align}
x^{k}y\cdot e &= x^{k} \cdot 7 = k+7, \\
yx^{k}\cdot e &= y\cdot(k+3).
\end{align}
The equality $x^{k}y\cdot e = yx^{k}\cdot e$ holds only if
$k+3$ is mapped to $k+7$ by $y$. By \eqref{eq:yqdef}, this occurs only when
\begin{align}
k + 3 &= 11v+3,\ 11v+8,\ \text{or}\ 3q-v,
\end{align}
that is,
\begin{align}
k &= 11v,\ 11v+5,\ \text{or}\ 3q-v-3.
\end{align}

Since $1 \le k \le 3q-1$, the case $v = 0$ is excluded from $k = 11v$. 

Assume that $k = 11v+5$. If $x^{k}y\cdot e = yx^{k}\cdot e$,
then we must have
\begin{align}
3q - v = 11v+8+4,
\end{align}
and hence
\begin{align}
\label{eq:qicond2}
v = \f{q-4}{4}.
\end{align}
On the other hand, by \eqref{eq:yqdef}, we have \eqref{eq:k11v7-2}.
Since \eqref{eq:k11v7-2} and \eqref{eq:qicond2} are incompatible, the case
$k = 11v+5$ is excluded.

Assume that $k = 3q-v-3$. If $x^{k}y\cdot e = yx^{k}\cdot e$,
then we must have
\begin{align}
11v+13 = 3q-v+4,
\end{align}
and hence
\begin{align}
v = \f{q-3}{4}.
\end{align}
This is incompatible with \eqref{eq:k11v7-2} as well. Hence $k = 3q-v-3$ is also excluded.

Thus, for $e = 3$, the equality $x^{k}y\cdot e = yx^{k}\cdot e$ holds only if
\begin{align}
\label{eq:ke2}
k = 11v\q (v \ge 1).
\end{align}

The conditions \eqref{eq:ke1} and \eqref{eq:ke2} cannot simultaneously hold for any $k$. Therefore
\eqref{eq:3qxkytrivial} holds.
Consequently, for the corresponding \dde\ $\msD$, we have
$\AD \cong \{ 1 \}$.
\end{proof}

\begin{prop}
\label{prop:class3beq3}
If a uniform passport $[3^{q}, 3^{q}, n]$ $(n = 3q)$ has genus at least~$2$, then it admits a \dde with a trivial automorphism group.
\end{prop}

\begin{proof}
By \eqref{eq:genusclass3}, we have
\begin{align}
\f{q+1}{2} \ge 2, \quad q+1 \equiv 0 \pmod{2}.
\end{align}
Hence $q \ge 3$ and $q$ is odd.

For such $q$, let us relabel the permutations $y$ and $(xy)^{-1}$ in \propref{prop:c3b3}
as $x$ and $y$, respectively.
Then $x$, $y$, and $(xy)^{-1}$ have cycle types $(3^{q})$, $(3^{q})$, and $(n)$, respectively.

Let the monodromy group $G = \gen{x, y}$. Then the corresponding dessin $\msD$ satisfies $\AD \cong \{ 1 \}$.
\end{proof}

\subsection{The Subcase $b \ge 5$}

The following lemma will be used in the proof of the case $b \ge 5$.

\begin{lemma}
\label{lem:gamma-at}
Let $\Gamma$ be the gamma function. If $a > 0$ and $t \ge 1$, then
\begin{align}
\Gamma(a+t+1) \le (a+t)^{t}\,\Gamma(a+1).
\end{align}
Equality holds if and only if $t=1$.
\end{lemma}

\begin{proof}
Define
\begin{align}
F(t)
&\ceq \log \frac{(a+t)^{t}\,\Gamma(a+1)}{\Gamma(a+t+1)} \\
&= t\log(a+t) + \log \Gamma(a+1) - \log \Gamma(a+t+1).
\end{align}
Then
\begin{align}
F(1)
&= \log \frac{(a+1)\Gamma(a+1)}{\Gamma(a+2)}
= \log \frac{\Gamma(a+2)}{\Gamma(a+2)}
= \log 1
= 0,
\end{align}
and
\begin{align}
\label{eq:gammaFt}
F'(t)
&= \log(a+t) + \frac{t}{a+t} - \psi(a+t+1),
\end{align}
where
\begin{align}
\psi(x)
\ceq \dv{x}\log \Gamma(x)
= \frac{\Gamma'(x)}{\Gamma(x)}
\end{align}
is the digamma function.

It is well known that
\begin{align}
\psi(x+1) < \log x + \frac{1}{x} \q (x>0).
\end{align}
Hence
\begin{align}
\psi(a+t+1) < \log(a+t) + \frac{1}{a+t}.
\end{align}
Therefore, by \eqref{eq:gammaFt},
\begin{align}
F'(t) &> \log(a+t) + \frac{t}{a+t} - \log(a+t) - \frac{1}{a+t} = \frac{t-1}{a+t}.
\end{align}
Since $t \ge 1$, we have $F'(t) \ge 0$, and moreover $F'(t) > 0$ whenever $t>1$.
Therefore, $F$ is increasing on $[1,\infty)$ and strictly increasing on $(1,\infty)$.
Hence
\begin{align}
F(t) \ge F(1) = 0.
\end{align}
Therefore,
\begin{align}
\log \frac{(a+t)^{t}\,\Gamma(a+1)}{\Gamma(a+t+1)} \ge 0,
\end{align}
and hence
\begin{align}
\Gamma(a+t+1) \le (a+t)^{t}\,\Gamma(a+1).
\end{align}
Equality holds if and only if $t=1$.
\end{proof}

\begin{prop}
\label{prop:class3bge5}
Every passport $[b^{q}, b^{q}, n]$ of genus~$\ge 2$ with $b \ge 5$ and $q \ge 2$ admits a dessin with a trivial automorphism group.
\end{prop}

\begin{proof}
By the argument in Section~\ref{sec:estimation}, it suffices to show that
\begin{align}
\vt{N(b, q, b)} > \vt{D(n)},
\end{align}
for $N$ and $D$ defined in \eqref{eq:tncd}.

By \lemref{lem:Nbqblower}, we have
\begin{align}
\vt{N(b, q, b)} \ge \f{(n-q)!}{2^{q(b-2)}\left(\f{q+1}{2}\right)!\left(\f{q-1}{2}\right)!}\left(\f{b-1}{2b}\right)^{\f{q-1}{2}},
\end{align}
and by \lemref{lem:Dupper}, since $n = bq$ is odd,
\begin{align}
\vt{D(n)} \le \kappa_{2}\cdot 3^{\f{n}{3}}\left(\f{n}{3}\right)!\q \left(\left(\f{n}{3}\right)! \ceq \Gamma\left(\f{n}{3}+1\right)\right),
\end{align}
where $\kappa_{2} = 972/947$.
Hence
\begin{align}
\f{\vt{N(b,q,b)}}{\vt{D(n)}} &\ge \f{(n-q)!}{2^{q(b-2)}\left(\f{q+1}{2}\right)!\left(\f{q-1}{2}\right)!}\left(\f{b-1}{2b}\right)^{\f{q-1}{2}}
\left(\kappa_{2}\cdot 3^{\f{n}{3}}\left(\f{n}{3}\right)!\right)^{-1} \\
&= \f{(q(b-1))!}{\kappa_{2}\cdot 2^{q(b-2)}3^{\f{bq}{3}}\left(\f{bq}{3}\right)!\left(\f{q+1}{2}\right)!\left(\f{q-1}{2}\right)!}\left(\f{b-1}{2b}\right)^{\f{q-1}{2}}.
\end{align}

Let the right-hand side be denoted by $R(b,q)$.
We prove that $R(b,q) > 1$ for all $b \ge 5$ and $q \ge 3$.
To this end, it suffices to show the following:
\begin{enumerate}
\item\label{itm:R53} $R(5,3) > 1$.
\item\label{itm:R5q} For $b=5$, $R(b,q)$ is increasing in $q \ge 3$.
\item\label{itm:Rbq} For each fixed $q \ge 3$, $R(b,q)$ is increasing in $b$.
\end{enumerate}
\HL

\item[\ref{itm:R53}] We have
\begin{align}
R(5, 3) = \f{(3(5-1))!}{\kappa_{2}\cdot 2^{3(5-2)}3^{\f{15}{3}}\left(\f{15}{3}\right)!\left(\f{3+1}{2}\right)!\left(\f{3-1}{2}\right)!}\left(\f{5-1}{10}\right)^{\f{3-1}{2}} = \f{72919}{11664} > 1.
\end{align}
\HL

\item[\ref{itm:R5q}] We compute
\begin{align}
\f{R(5,q+2)}{R(5,q)} &= \f{(4(q+2))!}{2^{3(q+2)}3^{\f{5(q+2)}{3}}\left(\f{5(q+2)}{3}\right)!\left(\f{q+3}{2}\right)!\left(\f{q+1}{2}\right)!}\left(\f{2}{5}\right)^{\f{q+1}{2}} \\
&\qquad \cdot \f{2^{3q}3^{\f{5q}{3}}\left(\f{5q}{3}\right)!\left(\f{q+1}{2}\right)!\left(\f{q-1}{2}\right)!}{(4q)!}\left(\f{2}{5}\right)^{-\f{q-1}{2}} \\
&= \f{\prod_{k=1}^{8}(4q+k)\left(\f{5q}{3}\right)!}{2^{6}\cdot 3^{\f{10}{3}}\left(\f{5q+10}{3}\right)!\f{q+1}{2}\cdot\f{q+3}{2}}\cdot \f{2}{5} \\
&= \f{\prod_{k=1}^{8}(4q+k)}{40(q+1)(q+3)}\cdot \f{\left(\f{5q}{3}\right)!}{3^{\f{10}{3}}\left(\f{5q+10}{3}\right)!}.
\end{align}
By \lemref{lem:gamma-at},
\begin{align}
\left(\f{5q+10}{3}\right)! = \Gamma\left(\f{5q}{3}+\f{10}{3}+1\right) \le \left(\f{5q+10}{3}\right)^{\f{10}{3}}\Gamma\left(\f{5q}{3}+1\right)
= \left(\f{5q+10}{3}\right)^{\f{10}{3}}\left(\f{5q}{3}\right)!.
\end{align}
Hence
\begin{align}
\f{\left(\f{5q}{3}\right)!}{3^{\f{10}{3}}\left(\f{5q+10}{3}\right)!} \ge \f{\left(\f{5q}{3}\right)!}{3^{\f{10}{3}}\left(\f{5q+10}{3}\right)^{\f{10}{3}}\left(\f{5q}{3}\right)!}
= \f{1}{(5q+10)^{\f{10}{3}}}.
\end{align}
Therefore,
\begin{align}
\f{R(5,q+2)}{R(5,q)} &\ge \f{\prod_{k=1}^{8}(4q+k)}{40(q+1)(q+3)}\cdot \f{1}{(5q+10)^{\f{10}{3}}} \\
&\ge \f{(4q+1)^{8}}{40(q+1)(q+3)(5q+10)^{\f{10}{3}}}
\end{align}
Let the right-hand side be denoted by $f(q)$. Then
\begin{align}
f(3) &= \f{13^{8}}{15\sqrt[3]{25}\cdot 10^{6}} = 18.598\ldots > 1, \\
f'(q) &= \f{(4q+1)^{7}(16q^{3}+152q^{2}+388q+261)}{7500\sqrt[3]{5}(q+1)^{2} (q+2)^{\f{13}{3}} (q+3)^{2}} > 0.
\end{align}
Therefore,
\begin{align}
\f{R(5,q+2)}{R(5,q)} > 1 \q (q \ge 3).
\end{align}
Hence $R(5,q)$ is increasing in $q$.
\HL

\item[\ref{itm:Rbq}] We have
\begin{align}
\f{R(b+2,q)}{R(b,q)} &= \f{(q(b+1))!}{2^{qb}3^{\f{(b+2)q}{3}}\left(\f{(b+2)q}{3}\right)!\left(\f{q+1}{2}\right)!\left(\f{q-1}{2}\right)!}\left(\f{b+1}{2(b+2)}\right)^{\f{q-1}{2}} \cdot \\
&\qquad\f{2^{q(b-2)}3^{\f{bq}{3}}\left(\f{bq}{3}\right)!\left(\f{q+1}{2}\right)!\left(\f{q-1}{2}\right)!}{(q(b-1))!}\left(\f{b-1}{2b}\right)^{-\f{q-1}{2}} \\
&= \f{\prod_{k=1}^{2q}(q(b-1)+k)}{4^{q}}\cdot\f{\left(\f{qb}{3}\right)!}{3^{\f{2q}{3}}\left(\f{q(b+2)}{3}\right)!}
\left(\f{b(b+1)}{(b-1)(b+2)}\right)^{\f{q-1}{2}}.
\end{align}
Since $2q/3 \ge 2 > 1$, by \lemref{lem:gamma-at} we have
\begin{align}
\left(\f{q(b+2)}{3}\right)! = \Gamma\left(\f{qb}{3}+\f{2q}{3}+1\right) \le \left(\f{q(b+2)}{3}\right)^{\f{2q}{3}}\Gamma\left(\f{qb}{3}+1\right)
= \left(\f{q(b+2)}{3}\right)^{\f{2q}{3}}\left(\f{qb}{3}\right)!.
\end{align}
Hence
\begin{align}
\f{\left(\f{qb}{3}\right)!}{3^{\f{2q}{3}}\left(\f{q(b+2)}{3}\right)!} \ge \f{\left(\f{qb}{3}\right)!}{3^{\f{2q}{3}}\left(\f{q(b+2)}{3}\right)^{\f{2q}{3}}\left(\f{qb}{3}\right)!}
= \f{1}{(q(b+2))^{\f{2q}{3}}}.
\end{align}
Therefore,
\begin{align}
\f{R(b+2,q)}{R(b,q)} &\ge \f{\prod_{k=1}^{2q}(q(b-1)+k)}{4^{q}}\cdot\f{1}{q(b+2)^{\f{2q}{3}}}
\left(\f{b(b+1)}{(b-1)(b+2)}\right)^{\f{q-1}{2}}.
\end{align}
Since
\begin{align}
&\prod_{k=1}^{2q}(q(b-1)+k) \ge (q(b-1)+1)^{2q}
\end{align}
and
\begin{align}
&b(b+1) - (b-1)(b+2) = 2 > 0,
\end{align}
we obtain
\begin{align}
\f{R(b+2,q)}{R(b,q)} \ge \f{(q(b-1)+1)^{2q}}{4^{q}(q(b+2))^{\f{2q}{3}}}\cdot 1^{\f{q-1}{2}}
= \left(\f{(q(b-1)+1)^{3}}{8q(b+2)}\right)^{\f{2q}{3}}.
\end{align}
Both the numerator and denominator of the right-hand side are positive. Define
\begin{align}
g(b, q) \ceq (q(b-1)+1)^{3} - 8q(b+2).
\end{align}
Then
\begin{align}
g(5, 3) &= 2029 > 0, \\
\pdv{g}{b} &= q(3(q(b-1)+1)^{2}-8) \ge q(3(3(5-1)+1)^{2}-8) = 499q > 0, \\
\pdv{g}{q} &= 3(b-1)(q(b-1)+1)^{2} - 8(b+2) \ge 3(b-1)(3(b-1)+1)^{2} - 8(b+2) \\
&= 27b^{3}-63b^{2}+40b-28 = 27(b-5)^{3} + 342\left(b-\f{1985}{684}\right)^{2}+\f{638471}{1368} > 0.
\end{align}
Hence $g(b,q) > 0$, and therefore
\begin{align}
\f{R(b+2,q)}{R(b,q)} > 1 \q (b \ge 5,\ q \ge 3).
\end{align}
Thus $R(b,q)$ is increasing in $b$.

By \ref{itm:R53}, \ref{itm:R5q}, and \ref{itm:Rbq}, we conclude that $R(b, q) > 1$
for all $b \ge 5$ and $q \ge 3$. Therefore,
\begin{align}
\vt{N(b, q, b)} > \vt{D(n)}.
\end{align}
This completes the proof.
\end{proof}

\begin{rem}
For $b = 3$, the above argument does not apply.

Indeed, when $b = 3$, we have
\begin{align}
R(3, 3) &= \f{(3(3-1))!}{\kappa_{2}\cdot 2^{3(3-2)}3^{\f{9}{3}}\left(\f{9}{3}\right)!\left(\f{3+1}{2}\right)!\left(\f{3-1}{2}\right)!}\left(\f{3-1}{6}\right)^{\f{3-1}{2}} = \f{4735}{52488} < 1,
\end{align}
and
\begin{align}
\f{R(3, q+2)}{R(3, q)} &= \f{(2(q+2))!}{2^{q+2}3^{q+2}(q+2)!\left(\f{q+3}{2}\right)!\left(\f{q+1}{2}\right)!}\left(\f{1}{3}\right)^{\f{q+1}{2}}
\cdot \f{2^{q}3^{q}q!\left(\f{q+1}{2}\right)!\left(\f{q-1}{2}\right)!}{(2q)!}\left(\f{1}{3}\right)^{-\f{q-1}{2}} \\
&= \f{4(2q+1)(2q+3)}{27(q+1)(q+3)} = \f{16}{27}\cdot\f{q+\f{1}{2}}{q+1}\cdot\f{q+\f{3}{2}}{q+3} < 1 \q (q \ge 3).
\end{align}
Hence $R(3, q)$ is decreasing in $q$, and therefore
\begin{align}
R(3, q) < 1 \q (q \ge 3).
\end{align}

Thus a different approach is required for the case $b = 3$.
\end{rem}

\subsection{Main theorem}

Combining the results of the previous two subsections, we obtain the following.

\begin{thm}
\label{thm:class3}
\thmclassthree
\end{thm}

\begin{proof}
If $q=1$, then $b=n$, and the statement follows from
\cite[Proposition~7.1]{Ohnishi26}.

Assume $q\ge2$.
Since the automorphism group is invariant under permutations of the three partitions,
it suffices to consider the passport $[b^{q}, b^{q}, n]$.

By \eqref{eq:genusclass3}, $b$ is odd and $b\ge3$.
The case $b=3$ is covered by \propref{prop:class3beq3},
and the case $b\ge5$ by \propref{prop:class3bge5}.
\end{proof}

\section{Counterexamples to Previous Conjectures}
\label{sec:counterex}

In \cite{Ohnishi26}, we conjectured the relationship between
the genus of uniform passports and their automorphism groups,
as summarized in \tabref{tab:genusautd}.
However, we found counterexamples to some of these conjectures.
In this section, we present two counterexamples concerning
passports of genus at least~2.

\subsection{No Dessins with Trivial Automorphism Group}

In \tabref{tab:genusautd}, we conjectured that every uniform passport of genus at least~2 admits
a \dde with trivial automorphism group.
However, by direct computation, we found the following counterexample.

\begin{itemize}
\item Passport: $[8^{2}, 2^{8}, 4^{4}]$
\item Genus: 2
\item With $x$ of cycle type $(8^{2})$ fixed, there are $920$ corresponding permutations $y$, and
every corresponding $\AD$ is nontrivial. \\
Up to conjugation by $C_{S_n}(x)$, these form $19$ conjugacy classes.
\end{itemize}

If $t \in C_{S_n}(x)$, then
$(x, tyt^{-1}) = (txt^{-1}, tyt^{-1})$ defines the same dessin as $(x,y)$.
Therefore, it suffices to count conjugacy classes of such permutations $y$
under conjugation by $C_{S_n}(x)$.

\tabref{tab:pass442882} lists representatives of the conjugacy classes of $y$
with $x = (1 \ldots 8)(9 \ldots 16)$ fixed, together with the corresponding orders of automorphism groups
and class sizes.
There are $11$, $6$, $1$, and $1$ conjugacy classes with
$\vt{\AD} = 2$, $4$, $8$, and $16$, respectively.

\begin{table}[htbp]
\centering
\small
$x = (1\ 2\ 3\ 4\ 5\ 6\ 7\ 8)(9\ 10\ 11\ 12\ 13\ 14\ 15\ 16)$ \\
\HL
\begin{tabular}{|c|c|c|}
\hline
Representative & $\vt{\AD}$ & Class size \\
\hline
$(1\ 3)(2\ 5)(4\ 6)(7\ 9)(8\ 10)(11\ 13)(12\ 15)(14\ 16)$ & 2 & 64 \\
\hline
$(1\ 3)(2\ 9)(4\ 6)(5\ 12)(7\ 14)(8\ 15)(10\ 16)(11\ 13)$ & 2 & 64 \\
\hline
$(1\ 3)(2\ 9)(4\ 6)(5\ 14)(7\ 11)(8\ 12)(10\ 16)(13\ 15)$ & 2 & 64 \\
\hline
$(1\ 3)(2\ 9)(4\ 15)(5\ 12)(6\ 13)(7\ 14)(8\ 11)(10\ 16)$ & 2 & 64 \\
\hline
$(1\ 4)(2\ 5)(3\ 9)(6\ 16)(7\ 13)(8\ 10)(11\ 14)(12\ 15)$ & 2 & 64 \\
\hline
$(1\ 4)(2\ 6)(3\ 9)(5\ 14)(7\ 16)(8\ 12)(10\ 13)(11\ 15)$ & 2 & 64 \\
\hline
$(1\ 4)(2\ 7)(3\ 9)(5\ 11)(6\ 14)(8\ 16)(10\ 13)(12\ 15)$ & 2 & 64 \\
\hline
$(1\ 4)(2\ 9)(3\ 7)(5\ 14)(6\ 10)(8\ 12)(11\ 15)(13\ 16)$ & 2 & 64 \\
\hline
$(1\ 4)(2\ 9)(3\ 14)(5\ 12)(6\ 15)(7\ 16)(8\ 11)(10\ 13)$ & 2 & 64 \\
\hline
$(1\ 9)(2\ 10)(3\ 12)(4\ 14)(5\ 15)(6\ 11)(7\ 16)(8\ 13)$ & 2 & 64 \\
\hline
$(1\ 9)(2\ 10)(3\ 13)(4\ 15)(5\ 11)(6\ 16)(7\ 12)(8\ 14)$ & 2 & 64 \\
\hline
$(1\ 3)(2\ 6)(4\ 9)(5\ 7)(8\ 13)(10\ 12)(11\ 15)(14\ 16)$ & 4 & 32 \\
\hline
$(1\ 3)(2\ 9)(4\ 11)(5\ 7)(6\ 13)(8\ 15)(10\ 16)(12\ 14)$ & 4 & 32 \\
\hline
$(1\ 3)(2\ 9)(4\ 15)(5\ 7)(6\ 13)(8\ 11)(10\ 16)(12\ 14)$ & 4 & 32 \\
\hline
$(1\ 5)(2\ 6)(3\ 9)(4\ 14)(7\ 13)(8\ 10)(11\ 15)(12\ 16)$ & 4 & 32 \\
\hline
$(1\ 5)(2\ 9)(3\ 12)(4\ 15)(6\ 13)(7\ 16)(8\ 11)(10\ 14)$ & 4 & 32 \\
\hline
$(1\ 9)(2\ 10)(3\ 11)(4\ 16)(5\ 13)(6\ 14)(7\ 15)(8\ 12)$ & 4 & 32 \\
\hline
$(1\ 5)(2\ 9)(3\ 7)(4\ 15)(6\ 13)(8\ 11)(10\ 14)(12\ 16)$ & 8 & 16 \\
\hline
$(1\ 9)(2\ 12)(3\ 15)(4\ 10)(5\ 13)(6\ 16)(7\ 11)(8\ 14)$ & 16 & 8 \\
\hline
\end{tabular}
\HL
\caption{Representatives of conjugacy classes of $y$ for the passport $[8^{2}, 2^{8}, 4^{4}]$}\label{tab:pass442882}
\end{table}

\figref{fig:dessin-notriv} shows a regular dessin with passport
$[4^{4}, 2^{8}, 8^{2}]$, obtained from
$[8^{2}, 2^{8}, 4^{4}]$ by interchanging the first and third partitions.

\begin{figure}[htbp]
\centering
\includegraphics[width=100mm]{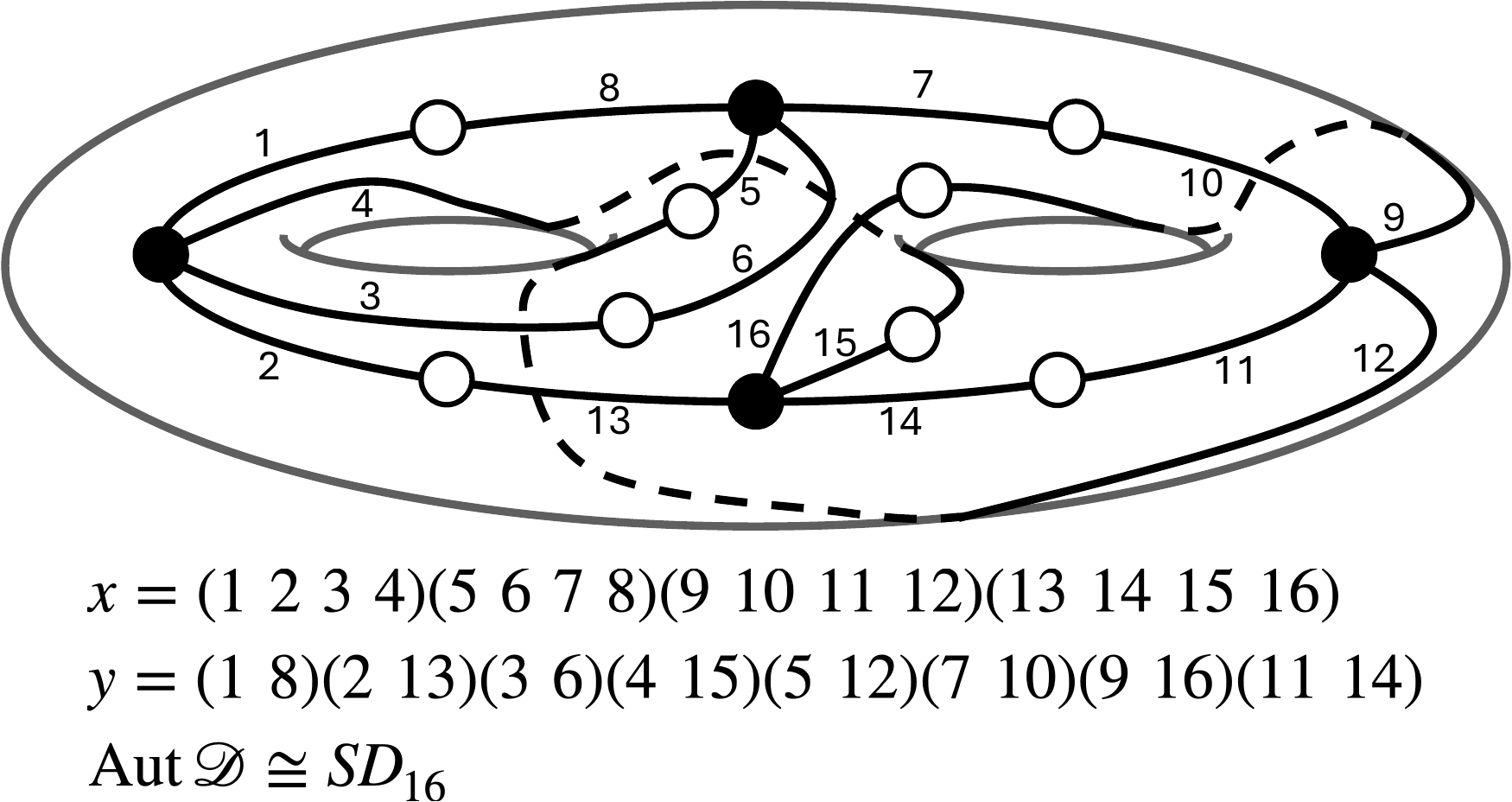} 
\caption{A regular \dde\ with passport $[4^{4}, 2^{8}, 8^{2}]$}
\label{fig:dessin-notriv}
\end{figure}

This counterexample initially suggested the conjecture that every dessin with passport
$[(4t)^{2}, 2^{4t}, (2t)^{4}]$ ($t \ge 2$)
has a nontrivial automorphism group.
However, this conjecture fails already for $t=3$, since the passport
$[12^{2}, 2^{12}, 6^{4}]$
admits a dessin with trivial automorphism group.
Determining the precise behavior of this family remains an open problem.

\subsection{Only Dessins with Trivial Automorphism Group}

In \tabref{tab:genusautd}, we conjectured that every uniform passport of genus at least~2
with at most one $(n)$-cycle admits
a non-regular \dde with nontrivial automorphism group,
that is, a dessin $\msD$ with $1 < \vt{\AD} < n$.
However, a direct computation revealed the following counterexample:

\begin{itemize}
\item Passport: $[15, 5^{3}, 5^{3}]$
\item Genus: 5
\item For a fixed $15$-cycle $x$, there are $4354560$ permutations $y$ ($=N(5,3,5)$),
and every corresponding dessin has trivial automorphism group.
\item The number of conjugacy classes of such $y$ is $290304$,
and each conjugacy class has $15$ elements.
\end{itemize}

\figref{fig:dessin-onlytriv} shows one of the dessins with passport
$[5^{3}, 5^{3}, 15]$, obtained from
$[15, 5^{3}, 5^{3}]$ by interchanging the first and third partitions.

\begin{figure}[htbp]
\centering
\includegraphics[width=100mm]{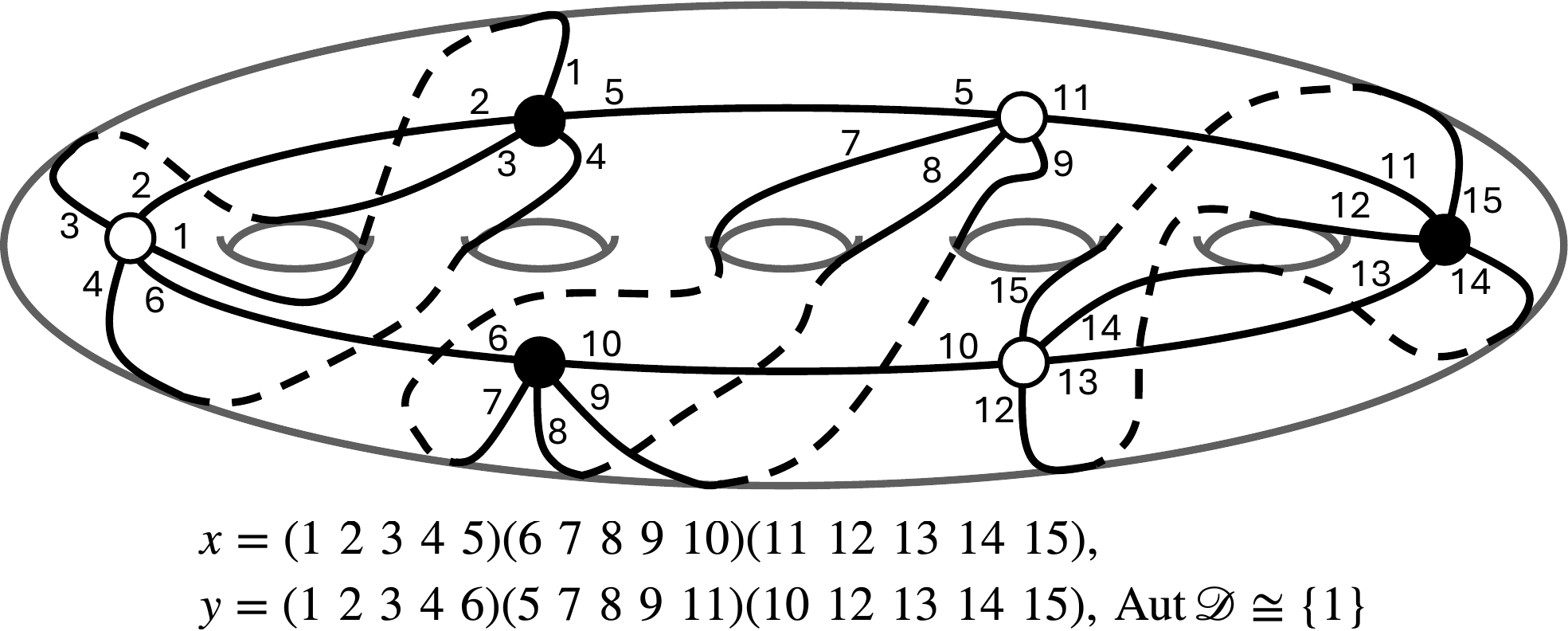} 
\caption{A \dde\ with passport $[5^{3}, 5^{3}, 15]$}
\label{fig:dessin-onlytriv}
\end{figure}

The above passports are special cases of the following theorem.

\begin{thm}
\label{thm:trivautd}
\thmtrivautd
\end{thm}

\begin{proof}
Let $n = \ell_{1}\ell_{2}$.
Since the automorphism group is invariant under permutations of the three partitions,
it suffices to consider the passport
$[\ell_{1}\ell_{2}, \ell_{2}^{\ell_{1}}, \ell_{2}^{\ell_{1}}]$.
Moreover, since the automorphism group is also invariant under relabeling of the edges,
we may assume that $x = \sigman$.

By \eqref{eq:genus}, the genus is
\begin{align}
\f{n-(2\ell_{1}+1)}{2}+1
=
\f{\ell_{1}(\ell_{2}-2)+1}{2}.
\end{align}
Since the genus is an integer, both $\ell_{1}$ and $\ell_{2}$ must be odd.
Hence $\ell_{1} \ge 3$, $\ell_{2} \ge 5$, and the genus is at least~5.

Let $y \in S_n$ be a permutation of cycle type
$(\ell_2^{\ell_1})$ corresponding to the first partition of the passport.
Then $(xy)^{-1}$ also has cycle type $(\ell_2^{\ell_1})$.
Since a permutation and its inverse have the same cycle type,
$xy$ also has cycle type $(\ell_2^{\ell_1})$.

Assume, for contradiction, that the dessin corresponding to $(x,y)$ has nontrivial automorphism group.

As in \propref{prop:D}, let $D$ be the set of permutations in $S_n$
such that the corresponding dessin has nontrivial automorphism group.
Then we have $y \in D$.

By \propref{prop:D},
\begin{align}
D = D_{\f{n}{\ell_{2}}} \cup D_{\f{n}{\ell_{1}}} = D_{\ell_{1}} \cup D_{\ell_{2}}.
\end{align}
Hence $y$ belongs to either $D_{\ell_1}$ or $D_{\ell_2}$.
\HL

\item[(i)] Assume that $y \in D_{\ell_{1}}$.

By the definition \eqref{eq:defDk}, we have
\begin{align}
\label{eq:ycomm1}
yx^{\ell_{1}} = x^{\ell_{1}}y.
\end{align}
For $e \in E = \{ 1, \ldots, n \}$, denote by $O_{1}(e)$ the orbit of $e$ under $x^{\ell_{1}}$.
Then
\begin{align}
O_{1}(e) = \{ e, e+\ell_{1}, e+2\ell_{1}, \ldots, e+(\ell_{2}-1)\ell_{1} \},
\end{align}
where the entries are taken modulo $n$. Therefore, $E$ is partitioned into the following
$\ell_1$ distinct orbits:
\begin{align}
B_{1} &\ceq O_{1}(1) = \{ 1, 1+\ell_{1}, \ldots, 1 + (\ell_{2}-1)\ell_{1} \}, \\
B_{2} &\ceq O_{1}(2) = \{ 2, 2+\ell_{1}, \ldots, 2 + (\ell_{2}-1)\ell_{1} \}, \\
&\cdots \\
B_{\ell_{1}} &\ceq O_{1}(\ell_{1}) = \{ \ell_{1}, 2\ell_{1}, \ldots, \ell_{2}\ell_{1} \}, \\
\end{align}

By \eqref{eq:ycomm1}, for $e \in E$ and $m \in \Z$ we have
\begin{align}
y\cdot((x^{\ell_1})^m\cdot e) = (x^{\ell_1})^m\cdot(y\cdot e).
\end{align}
Hence
\begin{align}
y\cdot O_{1}(e) = O_{1}(y \cdot e).
\end{align}
Therefore, for any $1 \le i \le \ell_{1}$, we have
$y\cdot B_i=B_j$ for some $1\le j\le \ell_1$.
In other words, $y$ permutes the subsets
$B_1,\dots,B_{\ell_1}$.
Thus the partition
\begin{align}
\mcB \ceq \{B_1,\dots,B_{\ell_1}\}
\end{align}
is preserved by $y$, and hence forms a block system for the action of $y$ on $E$.

Therefore, $y$ induces a permutation $\tau$ on the set of blocks $\mcB$, and hence
$\tau \in S_{\ell_{1}}$.
Since $y$ has cycle type $(\ell_{2}^{\ell_{1}})$, we have
\begin{align}
y^{\ell_{2}} = \id_{S_{n}},
\end{align}
hence
\begin{align}
\tau^{\ell_{2}} = \id_{S_{\ell_{1}}}.
\end{align}
Therefore, $\ord(\tau)\mid \ell_2$. Since $\ell_2$ is prime,
\begin{align}
\label{eq:tau2}
\ord(\tau) = 1\ \text{or}\ \ell_{2}.
\end{align}

On the other hand, since $\tau\in S_{\ell_1}$, we have
\begin{align}
\ord(\tau)\mid |S_{\ell_1}|=\ell_1!.
\end{align}
Since $\ell_2$ is a prime greater than $\ell_1$,
we have $\ell_2\nmid\ell_1!$.
Therefore $\ord(\tau)\ne\ell_2$, and hence
\begin{align}
\ord(\tau)=1.
\end{align}

Consequently,
\begin{align}
\label{eq:yblock}
y\cdot B_{i} = B_{i}\q(1 \le i \le \ell_{1}).
\end{align}

Since $xy$ also has cycle type $(\ell_{2}^{\ell_{1}})$, we have
\begin{align}
(xy)^{\ell_{2}} = \id_{S_{n}}.
\end{align}
Hence, on the block system $\mcB$,
\begin{align}
\label{eq:xyblock1}
(xy)^{\ell_{2}}\cdot B_{i} = B_{i}\q(1 \le i \le \ell_{1}).
\end{align}

On the other hand, by \eqref{eq:yblock}, $y$ preserves every block $B_i$.
Hence
\begin{align}
xy\cdot B_i = x\cdot(y\cdot B_i) = x\cdot B_i.
\end{align}
Therefore, the induced action of $xy$ on $\mcB$ coincides with that of $x$, and thus
\begin{align}
(xy)^{\ell_2}\cdot B_i = x^{\ell_2}\cdot B_i.
\end{align}
Let $B_{j} = x^{\ell_{2}}\cdot B_{i}$. Since $x$ cyclically permutes the blocks,
\begin{align}
x \cdot B_{i} =
\begin{dcases}
B_{i+1} & (1 \le i \le \ell_{1}-1) \\
B_{1} & (i = \ell_{1})
\end{dcases},
\end{align}
we obtain
\begin{align}
\label{eq:xyblock2}
j \equiv i + \ell_{2} \pmod{\ell_{1}}.
\end{align}

By \eqref{eq:xyblock1} and \eqref{eq:xyblock2},
\begin{align}
i \equiv i + \ell_{2} \pmod{\ell_{1}},
\end{align}
and hence $\ell_1\mid\ell_2$.
This contradicts the assumption that $\ell_1$ and $\ell_2$ are distinct primes.

Therefore $y \notin D_{\ell_{1}}$.
\HL

\item[(ii)] Assume that $y \in D_{\ell_{2}}$.

By the definition \eqref{eq:defDk}, we have
\begin{align}
\label{eq:ycomm2}
yx^{\ell_{2}} = x^{\ell_{2}}y.
\end{align}
For $e \in E = \{ 1, \ldots, n \}$, denote by $O_{2}(e)$ the orbit of $e$ under $x^{\ell_{2}}$.
Then
\begin{align}
O_{2}(e) = \{ e, e+\ell_{2}, e+2\ell_{2}, \ldots, e+(\ell_{1}-1)\ell_{2} \},
\end{align}
where the entries are taken modulo $n$. Therefore, $E$ is partitioned into the following
$\ell_2$ distinct orbits:
\begin{align}
C_{1} &\ceq O_{2}(1) = \{ 1, 1+\ell_{2}, \ldots, 1 + (\ell_{1}-1)\ell_{2} \}, \\
C_{2} &\ceq O_{2}(2) = \{ 2, 2+\ell_{2}, \ldots, 2 + (\ell_{1}-1)\ell_{2} \}, \\
&\cdots \\
C_{\ell_{2}} &\ceq O_{2}(\ell_{2}) = \{ \ell_{2}, 2\ell_{2}, \ldots, \ell_{1}\ell_{2} \}, \\
\end{align}

By \eqref{eq:ycomm2}, for $e \in E$ and $m \in \Z$ we have
\begin{align}
y\cdot((x^{\ell_2})^m\cdot e) = (x^{\ell_2})^m\cdot(y\cdot e).
\end{align}
Hence
\begin{align}
y\cdot O_{2}(e) = O_{2}(y \cdot e).
\end{align}
Therefore, for any $1 \le i \le \ell_{2}$, $y\cdot C_{i} = C_{j}$ for some $1 \le j \le \ell_{2}$.
Thus the set
\begin{align}
\mcC \ceq \{C_1,\dots, C_{\ell_2}\}
\end{align}
forms a block system preserved by $y$.

Therefore, $y$ induces a permutation $\pi$ on the set of blocks $\mcC$, and hence
$\pi \in S_{\ell_{2}}$.

For $e \in E$, write
\begin{align}
e = (i, t),
\end{align}
where $e \in C_{i}$ and $e = i + t\ell_{2} \pmod{n}$. Then $1 \le i \le \ell_{2}$ and
we may regard $t$ as an element of $\Z/\ell_1\Z$.
Then
\begin{align}
x^{\ell_{2}}\cdot(i, t) &= (i, t+1).
\end{align}
Therefore, there exists a map $f_{i}: \Z/\ell_1\Z \to \Z/\ell_1\Z$ such that
\begin{align}
\label{eq:yfit}
y\cdot(i,t)=(\pi(i),f_i(t)).
\end{align}

Moreover,
\begin{align}
yx^{\ell_{2}}\cdot(i, t) &= y\cdot(i,t+1) = (\pi(i), f_{i}(t+1)), \\
x^{\ell_{2}}y\cdot(i, t) &= x^{\ell_{2}}(\pi(i), f_{i}(t)) = (\pi(i), f_{i}(t)+1).
\end{align}
By \eqref{eq:ycomm2},
\begin{align}
(\pi(i), f_{i}(t+1)) = (\pi(i), f_{i}(t)+1),
\end{align}
hence
\begin{align}
f_{i}(t+1) = f_{i}(t)+1\q(t \in \Z/\ell_{1}\Z).
\end{align}
Let $s_{i} = f_{i}(0)$. Then we obtain
\begin{align}
f_{i}(1) &= f_{i}(0)+1 = s_{i} + 1, \\
f_{i}(2) &= f_{i}(1)+1 = s_{i} + 2, \\
&\cdots \\
f_{i}(\ell_{1}-1) &= f_{i}(\ell_{1}-2)+1 = s_{i} + (\ell_{1}-1).
\end{align}
Hence
\begin{align}
f_{i}(t) = s_{i} + t \q (t \in \Z/\ell_{1}\Z).
\end{align}
Substituting this into \eqref{eq:yfit}, we obtain
\begin{align}
\label{eq:ypi}
y\cdot(i, t) = (\pi(i), s_{i}+t).
\end{align}
Thus $y$ maps each block $C_i$ to $C_{\pi(i)}$ and acts on the second coordinate
by a translation modulo $\ell_1$.

Since $y$ has cycle type $(\ell_2^{\ell_1})$, we have
\begin{align}
\label{eq:yell2}
y^{\ell_2}=\id_{S_{n}}.
\end{align}
Since $\pi$ is the permutation induced by $y$ on the block system $\mcC$, it follows that
\begin{align}
\pi^{\ell_2}=\id_{S_{\ell_2}}.
\end{align}
Hence $\ord(\pi)\mid \ell_2$. Since $\ell_2$ is prime, we have
\begin{align}
\ord(\pi) = 1\ \text{or}\ \ell_{2}.
\end{align}

If $\ord(\pi) = 1$, then $\pi = \id_{S_{\ell_{2}}}$ and hence
\begin{align}
y\cdot(i, t) = (i, s_{i}+t).
\end{align}
Thus $y$ preserves each block $C_i$.
Therefore every cycle of $y$ is contained in a block of size $\ell_{1}$,
and hence has length at most $\ell_1$.
This contradicts the fact that every cycle of $y$ has length $\ell_2$.
Therefore,
\begin{align}
\ord(\pi) = \ell_{2}.
\end{align}

By \eqref{eq:yell2}, we have
\begin{align}
\label{eq:yell2-1}
y^{\ell_{2}}\cdot(i, t) = (i, t).
\end{align}
Since $\ord(\pi)=\ell_2$,
the orbit of every block under $\pi$ consists of all blocks in $\mcC$.
Hence, starting from $(i, t)$ and applying $y$ repeatedly,
the block component visits each block $C_{j}$ exactly once before returning to $C_{i}$.
Therefore each translation parameter $s_j$ is added exactly once, and hence
\begin{align}
\label{eq:yell2-2}
y^{\ell_2}\cdot(i,t) = \left(i,\ \sum_{j=1}^{\ell_2}s_j+t\right).
\end{align}
By \eqref{eq:yell2-1} and \eqref{eq:yell2-2},
\begin{align}
\label{eq:sumsi1}
\sum_{j=1}^{\ell_{2}}s_{j} \equiv 0 \pmod{\ell_{1}}.
\end{align}

The action of $x$ is given by
\begin{align}
x\cdot (i, t) =
\begin{dcases}
(i+1, t) & (i \ne \ell_{2}) \\
(1, t+1) & (i = \ell_{2})
\end{dcases}.
\end{align}
Hence, together with \eqref{eq:ypi}, we obtain
\begin{align}
xy \cdot (i, t) = 
\begin{dcases}
(\pi(i)+1, s_{i} + t) & (\pi(i) \ne \ell_{2}) \\
(1, s_{i} + t + 1) & (\pi(i) = \ell_{2})
\end{dcases}.
\end{align}

Since $xy$ has cycle type $(\ell_{2})^{\ell_{1}}$, we have
\begin{align}
(xy)^{\ell_{2}} = \id_{S_{n}},
\end{align}
and hence
\begin{align}
\label{eq:xyell2-1}
(xy)^{\ell_{2}}\cdot(i, t) = (i, t).
\end{align}
Define a permutation $\pi'$ by
\begin{align}
\pi'(i) =
\begin{dcases}
\pi(i) + 1 & (\pi(i) \ne \ell_{2}) \\
1 & (\pi(i) = \ell_{2})
\end{dcases},
\end{align}
equivalently,
\begin{align}
\pi'(i) =
\begin{dcases}
\pi(i) + 1 & (i \ne \pi^{-1}(\ell_{2})) \\
1 & (i = \pi^{-1}(\ell_{2}))
\end{dcases},
\end{align}
then
\begin{align}
xy\cdot (i, t) =
\begin{dcases}
(\pi'(i), s_{i}+t) & (i \ne \pi^{-1}(\ell_{2})) \\
(\pi'(i), s_{i}+t+1) & (i = \pi^{-1}(\ell_{2})).
\end{dcases}
\end{align}
Clearly, $\pi' \in S_{\ell_2}$.
Since $(xy)^{\ell_2}=\id_{S_n}$ and $\pi'$ is the permutation induced by
$xy$ on the block system $\mcC$, we have
\begin{align}
(\pi')^{\ell_2}=\id_{S_{\ell_2}}.
\end{align}
Hence $\ord(\pi')\mid\ell_2$.
Since $\ell_2$ is prime,
\begin{align}
\ord(\pi')=1\ \text{or}\ \ell_{2}.
\end{align}
As in the case of $\pi$, the equality $\ord(\pi')=1$ leads to a contradiction.
Therefore,
\begin{align}
\ord(\pi')=\ell_2.
\end{align}

Thus, the orbit of $(i,t)$ under $xy$ visits all blocks $C_i$
exactly once before returning to the initial block.
Therefore, the translations by the $s_{j}$ are accumulated once each, and
in addition, $1$ is added once, namely when the orbit passes through
the block $C_{\pi^{-1}(\ell_2)}$. Hence
\begin{align}
\label{eq:xyell2-2}
(xy)^{\ell_{2}}\cdot(i, t) = \left(i,\ \sum_{j=1}^{\ell_{2}}s_{j} + 1 + t\right).
\end{align}
By \eqref{eq:xyell2-1} and \eqref{eq:xyell2-2},
\begin{align}
\label{eq:sumsi2}
\sum_{j=1}^{\ell_{2}}s_{j} + 1 \equiv 0 \pmod{\ell_{1}}.
\end{align}
This contradicts \eqref{eq:sumsi1}.

Therefore $y \notin D_{\ell_{2}}$.
\HL

By (i) and (ii), neither $y\in D_{\ell_{1}}$ nor $y\in D_{\ell_{2}}$ is possible. Since
$D = D_{\ell_{1}} \cup D_{\ell_{2}}$,
we have $y \notin D$.
Therefore the corresponding dessin has trivial automorphism group.
\end{proof}

Accordingly, in addition to $[15,5^{3},5^{3}]$, each of the passports
\begin{align}
[21,7^{3},7^{3}],\ [33,11^{3},11^{3}],\ [35,7^{5},7^{5}],\ [39,13^{3},13^{3}],\ [51,17^{3},17^{3}],\ [55,11^{5},11^{5}],\ldots
\end{align}
has the property that every corresponding dessin has trivial automorphism group.

\section{Passport $[n, n, n]$ with Genus $\ge 2$}
\label{sec:class1}

In this section, we consider passports of the form $[n, n, n]$ with genus at least~2.
A dessin with such a passport has one black vertex, one white vertex, and one face.

By \eqref{eq:genus}, the genus is
\begin{align}
\label{eq:genusclass1}
g=\f{n-3}{2}+1=\f{n-1}{2} \ge 2.
\end{align}
Since $g$ is an integer, $n$ is odd and $n \ge 5$.

By \cite[Corollary~6.2]{Ohnishi26}, this passport always admits a regular dessin.
On the other hand, by \cite[Proposition~7.1]{Ohnishi26}, it always admits a dessin with trivial automorphism group.
We investigate the intermediate case, namely whether the passport admits a nonregular dessin
with nontrivial automorphism group.

\subsection{Nonregular Dessins with Nontrivial Automorphism Group}

\begin{thm}
\label{thm:class1}
\thmclassone
\end{thm}

\begin{proof}
Assume that $n$ is prime.
Since $\vt{\AD}$ divides the number of edges $n$, we have $\vt{\AD}=1$ or $n$.
If $\vt{\AD}=1$, then the dessin has trivial automorphism group.
If $\vt{\AD}=n$, then the dessin is regular by \lemref{lem:order-n}.
Hence this passport admits no nonregular dessin with nontrivial automorphism group.

Thus, the existence of a nonregular dessin with nontrivial automorphism group implies that $n$ is composite.

Conversely, assume that $n$ is composite. Let $r$ be a divisor of $n$ satisfying $1<r<n$.
Then we may write $n=rs$. Since $n$ is odd, both $r$ and $s$ are odd integers, and $r, s \ge 3$.

For a dessin $\msD$ with the passport $[n, n, n]$, fix $x = (1\ 2\ \ldots\ n)$ and let
\begin{align}
t &= (s\ 2s\ \ldots\ rs(= n)), \\
y &= txt^{-1}.
\end{align}
Explicitly, $y$ is obtained from $x$ by cyclically permuting the entries
$s,2s,\ldots,(r-1)s,rs(=n)$. Hence $y$ also has cycle type $(n)$.

We show that
\begin{enumerate}
\item \label{itm:c1xy} $(xy)^{-1}$ has cycle type $(n)$, and
\item \label{itm:c1AD} For the dessin $\msD$ corresponding to $x$ and $y$, we have $\vt{\AD}=r$.
\end{enumerate}

In what follows, all indices are taken modulo $n$.
\HL

\item[\ref{itm:c1xy}]
Since $xy$ and $(xy)^{-1}$ have the same cycle type, it suffices to show that $xy$ has cycle type $(n)$.
For each $e \in E = \{ 1, \ldots, n \}$, we can uniquely write $e = is+j$ with $0 \le i \le r-1$ and $1 \le j \le s$.

If $j\ne s$, then $e$ is fixed by both $t$ and $t^{-1}$.
On the other hand,
$t$ sends $is+s$ to $(i+1)s+s$, while $t^{-1}$ sends
$is+s$ to $(i-1)s+s$.

Moreover, $x$ sends $is+j$ to $is+j+1$.

Therefore, for $1\le j\le s-2$, the action of $xy=xtxt^{-1}$ on $e$ is
\begin{align}
is+j \onarrow{t^{-1}} is+j \onarrow{x} is+j+1 \onarrow{t} is+j+1 \onarrow{x} is+j+2,
\end{align}
hence $xy\cdot e = e+2$.

For $j = s-1$, we have
\begin{align}
is+s-1 \onarrow{t^{-1}} is+s-1 \onarrow{x} is+s \onarrow{t} (i+1)s+s \onarrow{x} (i+2)s+1,
\end{align}
hence $xy\cdot e = e+s+2$.

For $j = s$, we have
\begin{align}
is+s \onarrow{t^{-1}} (i-1)s+s \onarrow{x} is+1 \onarrow{t} is+1 \onarrow{x} is+2,
\end{align}
hence $xy\cdot e = e-s+2$.

Thus, since $s$ is odd and $s \ge 3$, starting from $is+1$ and applying $xy$
repeatedly, we obtain
\begin{align}
is+1 &\onarrow{xy} is+3 \rightarrow \cdots \rightarrow is+s-2 \rightarrow is+s \\
&\rightarrow is+2 \rightarrow is+4 \rightarrow \cdots \rightarrow is+s-1 \rightarrow (i+2)s+1.
\end{align}

Each element $is+j$ ($1\le j\le s$) appears exactly once in this sequence,
and the last element is mapped to $(i+2)s+1$.

Since $\gcd(r,2)=1$, repeated addition of $2$ modulo $r$
visits every residue class modulo $r$.
Hence the orbit passes through all elements of $E$
before returning to its starting point.

Therefore, $xy$ has cycle type $(n)$.
\HL

\item[\ref{itm:c1AD}]
By \corref{cor:ADtrivial}\ref{itm:ADtrivial1}, we have
\begin{align}
\label{eq:ADset}
\AD = \{ x^{k} \mid 0 \le k \le n-1,\ x^{k}y = yx^{k} \}.
\end{align}
$x^{k}$ sends $is+j$ to $is+j+k$ modulo $n$.

Assume that $1 \le k \le s-1$ and let $e = rs = n$. The image of $e$ under $x^{k} y=x^{k}txt^{-1}$ is
\begin{align}
rs \onarrow{t^{-1}} (r-1)s \onarrow{x} (r-1)s+1 \onarrow{t} (r-1)s+1 \onarrow{x^{k}} (r-1)s+k+1.
\end{align}
On the other hand, the image of $e$ under $yx^{k} = txt^{-1}x^{k}$ is
\begin{align}
rs \onarrow{x^{k}} rs+k  \onarrow{t^{-1}} rs+k \onarrow{x} rs+k+1 \onarrow{t}
\begin{dcases}
rs+k+1 & (k \ne s-1) \\
2s & (k = s-1)
\end{dcases}.
\end{align}
Since
\begin{align}
(r-1)s+k+1 \ne rs+k+1 \q (k \ne s-1)
\end{align}
and, when $k = s-1$,
\begin{align}
(r-1)s+k+1=(r-1)s+s=rs \ne 2s,
\end{align}
the images of $e$ under $x^{k}y$ and $yx^{k}$ differ. 
Therefore,
\begin{align}
\label{eq:xky1}
x^{k}y \ne yx^{k} \q (1 \le k \le s-1).
\end{align}

We next show that $x^{s}$ commutes with $y$.
To this end, consider the actions of
$x^{s}y=x^{s}txt^{-1}$ and $yx^{s}=txt^{-1}x^{s}$
on $e=is+j$ ($0 \le i \le r-1$, $1 \le j \le s$).

For $1 \le j \le s-2$, we have
\begin{align}
x^{s}y\colon &is+j \onarrow{t^{-1}} is+j \onarrow{x} is+j+1 \onarrow{t} is+j+1 \\
&\onarrow{x^{s}} (i+1)s+j+1, \\
yx^{s}\colon &is+j \onarrow{x^{s}} (i+1)s+j \\
&\onarrow{t^{-1}} (i+1)s+j \onarrow{x} (i+1)s+j+1 \onarrow{t} (i+1)s+j+1.
\end{align}

For $j = s-1$, we have
\begin{align}
x^{s}y\colon &is+s-1 \onarrow{t^{-1}} is+s-1 \onarrow{x} is+s \onarrow{t} (i+1)s+s \\
&\onarrow{x^{s}} (i+2)s+s, \\
yx^{s}\colon &is+s-1 \onarrow{x^{s}} (i+1)s+s-1 \\
&\onarrow{t^{-1}} (i+1)s+s-1 \onarrow{x} (i+1)s+s \onarrow{t} (i+2)s+s.
\end{align}

For $j = s$, we have
\begin{align}
x^{s}y\colon &is+s \onarrow{t^{-1}} (i-1)s+s \onarrow{x} is+1 \onarrow{t} is+1 \\
&\onarrow{x^{s}} (i+1)s+1, \\
yx^{s}\colon &is+s \onarrow{x^{s}} (i+1)s+s \\
&\onarrow{t^{-1}} is+s \onarrow{x} (i+1)s+1 \onarrow{t} (i+1)s+1.
\end{align}

Thus, $x^{s}y$ and $yx^{s}$ agree on every element of $E$. Therefore,
\begin{align}
\label{eq:xky2}
x^{s}y = yx^{s}.
\end{align}

For $0 \le k \le n-1 = rs-1$, write $k = us+v$ ($0 \le u \le r-1$, $0 \le v \le s-1$).
When $v=0$, by \eqref{eq:xky2},
\begin{align}
x^{k}y = (x^{s})^{u}y = y(x^{s})^{u} = yx^{k}.
\end{align}
On the other hand, when $1 \le v \le s-1$, by \eqref{eq:xky1} we have
$x^{v}y \ne yx^{v}$. Since right multiplication by $(x^{s})^{u}$ is bijective, we obtain
\begin{align}
x^{k}y = x^{v}(x^{s})^{u}y = x^{v}y(x^{s})^{u} \ne yx^{v}(x^{s})^{u} = yx^{k}.
\end{align}
Hence we obtain
\begin{align}
x^{k}y = yx^{k} \iff s \mid k.
\end{align}
Then by \eqref{eq:ADset},
\begin{align}
\vt{\AD} = \f{n}{s} = r.
\end{align}

Since $1<\vt{\AD}=r<n$,
the dessin $\msD$ is nonregular and has a nontrivial automorphism group.

Since $r$ was arbitrary, the second assertion of the theorem also follows.
\end{proof}

\figref{fig:dessin-class1} shows a \dde\ with passport $[9,9,9]$
of genus $4$ and $\vt{\AD}=3$.

\begin{figure}[htbp]
\centering
\includegraphics[width=125mm]{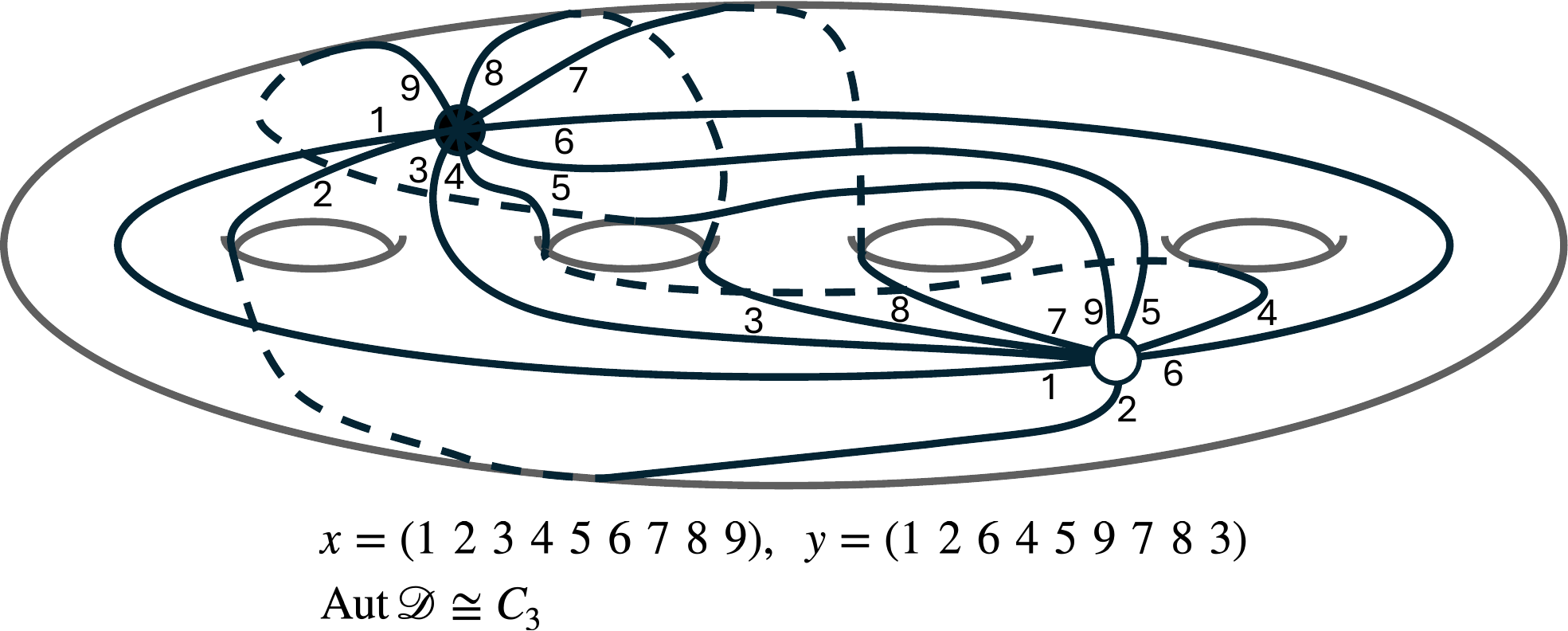}
\caption{A \dde\ with passport $[9,9,9]$ and $\vt{\AD}=3$}
\label{fig:dessin-class1}
\end{figure}

Combining the above theorem with the cases of regular dessins and dessins
with trivial automorphism group, we obtain the following corollary.

\begin{cor}
Let $[n,n,n]$ be a uniform passport of genus at least~2.
Then it admits a \dde\ $\msD$ with $\vt{\AD}=r$
if and only if $r$ is a divisor of $n$.
\end{cor}

\section{Alternative Proof of the $[n,b^{q},n]$ Case}
\label{sec:class2}

In \cite[Theorem~7.6]{Ohnishi26}, we proved that every passport $[n, b^{q}, n]$ with $n = bq$ and genus at least~$2$
admits a dessin with trivial automorphism group. In this section, we present an alternative proof using estimates for
the size of the centralizer $C_{S_n}(G)$.

For the passport $[n, n, n]$, that is, the case $q = 1$, \cite[Proposition~7.1]{Ohnishi26} gives an explicit construction
of $x$ and $y$ with $\AD \cong \{ 1 \}$.

The following two subsections treat separately
the cases $b = 2$ (where necessarily $q \ge 4$)
and $b \ge 3$ with $q \ge 2$.

\subsection{The Subcase $b = 2$}

We proved the $b=2$ case in \cite[Proposition~7.3]{Ohnishi26}.
Here we give another proof by explicitly constructing a pair $(x,y)$ for each $q$.

\begin{prop}
\label{prop:c2b2}
For each even integer $q \ge 4$, fix $x = (1\ 2\ \ldots\ 2q) \in S_{2q}$ and define $y \in S_{2q}$ for each $q$ as follows:
\begin{align}
\begin{aligned}
\label{eq:yqdef2}
&y = \A_{1}\A_{2}\cdots \A_{q/2-2}\B, \\
&\A_{i} = (4i{-}3\ 4i{-}1)(4i{-}2\ 4i)\q (1 \le i \le q/2-2), \\
&\B = (2q{-}7\ 2q{-}5)(2q{-}6\ 2q{-}3)(2q{-}4\ 2q{-}1)(2q{-}2\ 2q).
\end{aligned}
\end{align}
Then the following hold:
\begin{enumerate}
\item\label{itm:c2b2-1} $y$ has cycle type $(2^{q})$.
\item\label{itm:c2b2-2} $(xy)^{-1}$ has cycle type $(2q)$.
\item\label{itm:c2b2-3} For the \dde $\msD$ corresponding to $x$ and $y$, we have $\AD \cong \{ 1 \}$.
\end{enumerate}
\end{prop}

\begin{rem}
For $q = 4, 6, 8,$ and $10$, the permutations $y$ and $(xy)^{-1}$ are as follows:
\begin{itemize}
\item $q=4$:
\begin{align}
y&=(1\ 3)(2\ 5)(4\ 7)(6\ 8), \\
(xy)^{-1}&=(1\ 6\ 2\ 3\ 5\ 7\ 8\ 4)
\end{align}
\item $q=6$:
\begin{align}
y&=(1\ 3)(2\ 4)(5\ 7)(6\ 9)(8\ 11)(10\ 12), \\
(xy)^{-1}&=(1\ 10\ 6\ 7\ 9\ 11\ 12\ 8\ 5\ 2\ 3\ 4)
\end{align}
\item $q=8$:
\begin{align}
y&=(1\ 3)(2\ 4)(5\ 7)(6\ 8)(9\ 11)(10\ 13)(12\ 15)(14\ 16), \\
(xy)^{-1}&=(1\ 14\ 10\ 11\ 13\ 15\ 16\ 12\ 9\ 6\ 7\ 8\ 5\ 2\ 3\ 4)
\end{align}
\item $q=10$:
\begin{align}
y &= (1\ 3)(2\ 4)(5\ 7)(6\ 8)(9\ 11)(10\ 12)(13\ 15)(14\ 17)(16\ 19)(18\ 20), \\
(xy)^{-1}&= (1\ 18\ 14\ 15\ 17\ 19\ 20\ 16\ 13\ 10\ 11\ 12\ 9\ 6\ 7\ 8\ 5\ 2\ 3\ 4) \\
\end{align}
\end{itemize}
These examples illustrate the structure of the construction.
\end{rem}

\begin{proof}
\item[\ref{itm:c2b2-1}]
In \eqref{eq:yqdef2}, the permutation $y$ is written as a product of $q$ transpositions.
To show that $y$ has cycle type $(2^{q})$, it suffices to verify that each element of $\{1,\ldots,2q\}$ appears.

For each $s \in \{1,\ldots,2q\}$, we distinguish two cases:
\begin{itemize}
\item If $1 \le s \le 2q-8$, let $i = \fl{(s+3)/4}$. Then $1 \le i \le q/2-2$, and $s$ appears in $\A_i$ as one of
$4i-3$, $4i-2$, $4i-1$, or $4i$, according to $s \equiv 1,2,3,0 \pmod{4}$.
\item If $2q-7 \le s \le 2q$, then $s$ appears in $\B$.
\end{itemize}

Thus every element of $\{1,\ldots,2q\}$ appears, and hence $y$ has cycle type $(2^{q})$.
\HL

\item[\ref{itm:c2b2-2}]
Since $xy$ and $(xy)^{-1}$ have the same cycle type, it suffices to show that $xy$ has cycle type $(2^{q})$.
We now prove this.

Table~\ref{tab:xycycles2} displays a $2q$-cycle of $xy$ arising from \eqref{eq:yqdef2}.

\begin{table}[htbp]
\centering
\small
\begin{tabular}{|c|c|c|c|}
\hline
$\A / \B$ & $s$ & $y\cdot s$ & $xy\cdot s$ \\
\hline
\multirow{4}{*}{$\A_{i}$} & $4i-3$ & $4i-1$ & $4i$ \\
& $4i$ & $4i-2$ & $4i-1$ \\
& $4i-1$ & $4i-3$ & $4i-2$ \\
& $4i-2$ & $4i$ & $4i+1=4(i+1)-3$ \\
\hline
\multirow{8}{*}{$\B$} & $2q-7$ & $2q-5$ & $2q-4$ \\
& $2q-4$ & $2q-1$ & $2q$ \\
& $2q$ & $2q-2$ & $2q-1$ \\
& $2q-1$ & $2q-4$ & $2q-3$ \\
& $2q-3$ & $2q-6$ & $2q-5$ \\
& $2q-5$ & $2q-7$ & $2q-6$ \\
& $2q-6$ & $2q-3$ & $2q-2$ \\
& $2q-2$ & $2q$ & $1$ \\
\hline
\end{tabular}
\vspace{1\baselineskip}
\caption{A $2q$-cycle of $xy$ ($s = 1, \ldots, 2q$)}\label{tab:xycycles2}
\end{table}

In each $\A_{i}$, the four elements are permuted as follows:
\begin{align}
4i{-}3 \to 4i \to 4i{-}1 \to 4i{-}2 \to 4(i{+}1){-}3 \to \cdots,
\end{align}
where the last element $4(i+1)-3$ lies in the next block. 

If $i < q/2-2$, then the next block is $\A_{i+1}$, and the same pattern continues.

If $i = q/2-2$, then the next block is $\B$, and since $4(i+1)-3 = 2q-7$, the permutation continues as
\begin{align}
2q{-}7 \to 2q{-}4 \to 2q \to 2q{-}1 \to 2q{-}3 \to 2q{-}5 \to 2q{-}6 \to 2q{-}2 \to 1.
\end{align}

Thus all $2q$ elements are permuted in a single cycle by $xy$. Hence $xy$ has cycle type $(2q)$.
\HL

\item[\ref{itm:c2b2-3}]
By \corref{cor:ADtrivial}\ref{itm:ADtrivial2}, it suffices to show that
\begin{align}
\text{for every}\ 1 \le k \le 2q-1,\ x^{k}y \ne yx^{k}.
\end{align}
Equivalently,
\begin{align}
&\text{for every}\ 1 \le k \le 2q-1,\\
\label{eq:2qxkytrivial}
&\q\text{there exists}\ e \in E = \{ 1, \ldots, 2q \}\ \text{such that}\ x^{k}y\cdot e \ne yx^{k}\cdot e.
\end{align}

In what follows, the elements of $E$ are taken modulo $2q$.
\HL

\item[(i)] Let $e = 2q-3$. Then
\begin{align}
x^{k}y\cdot e &= x^{k} \cdot (2q-6) = k-6, \\
yx^{k}\cdot e &= y\cdot(k-3).
\end{align}
The equality $x^{k}y\cdot e = yx^{k}\cdot e$ holds only if
$k-3$ is mapped to $k-6$ by $y$. By \eqref{eq:yqdef2}, this occurs only when
\begin{align}
k -3 = 2q-3\ \text{or}\ 2q-1,
\end{align}
that is,
\begin{align}
k = 2q\ \text{or}\ 2.
\end{align}
Since $1 \le k \le 2q-1$, the case $k = 2q$ is excluded. Hence
\begin{align}
\label{eq:2qxky1}
x^{k}y\cdot e = yx^{k}\cdot e \iff k =2.
\end{align}

\item[(ii)] Let $e = 2q-1$. Then
\begin{align}
x^{k}y\cdot e &= x^{k} \cdot (2q-4) = k-4, \\
yx^{k}\cdot e &= y\cdot(k-1).
\end{align}
The equality $x^{k}y\cdot e = yx^{k}\cdot e$ holds only if
$k-1$ is mapped to $k-4$ by $y$. By \eqref{eq:yqdef2}, this occurs only when
\begin{align}
k - 1 = 2q-3\ \text{or}\ 2q-1,
\end{align}
that is,
\begin{align}
k = 2q-2\ \text{or}\ 2q.
\end{align}
Since $1 \le k \le 2q-1$, the case $k = 2q$ is excluded. Hence
\begin{align}
\label{eq:2qxky2}
x^{k}y\cdot e = yx^{k}\cdot e \iff k =2q-2.
\end{align}

\eqref{eq:2qxky1} and \eqref{eq:2qxky2} cannot simultaneously hold for any $k$.
Therefore \eqref{eq:2qxkytrivial} holds.
Consequently, for the corresponding \dde\ $\msD$, we have
$\AD \cong \{ 1 \}$.
\end{proof}

\subsection{The Subcase $b \ge 3$}

The following proposition was proved in \cite[Proposition~7.5]{Ohnishi26}.
Here we give an alternative proof.

\begin{prop}
\label{prop:c2bge3}
Every passport $[n, b^{q}, n]$ of genus~$\ge 2$ with $b \ge 3$ and $q \ge 2$ admits a dessin with a trivial automorphism group.
\end{prop}

\begin{proof}
By \eqref{eq:genus}, the genus is 
\begin{align}
\f{n-(1+q+1)}{2}+1 = \f{bq-(q+2)}{2}+1 = \f{q(b-1)}{2}.
\end{align}
Since $b \ge 3$ and $q \ge 2$, we have $q(b-1)/2 \ge 2$.
Moreover, since the genus is an integer, if $q$ is odd then $b$ is odd, and hence $n=bq$ is also odd.
\HL

By \thmref{thm:MN}, we have
\begin{align}
\label{eq:Nlower-e}
\f{\vt{N}}{\vt{T}} \ge \f{2}{n+2} = \f{2}{n(1+\f{2}{n})} \ge \f{2}{n(1+\f{2}{6})} = \f{3}{2n}.
\end{align}

On the other hand, by \propref{prop:tbq} and \lemref{lem:Dupper}, we have
\begin{align}
\f{\vt{D}}{\vt{T}} \le \kappa_{1}\cdot 2^{\f{n}{2}}\left(\f{n}{2}\right)!\cdot\f{b^{q}q!}{n!}\q
\left(\left(\f{n}{2}\right)! \ceq \Gamma\left(\f{n}{2}+1\right)\right),
\end{align}
where $\kappa_{1} = 2623/1894$.

Combining these inequalities, we obtain
\begin{align}
\f{\vt{N}}{\vt{D}} \ge \f{3}{2n}\cdot\f{n!}{\kappa_{1}\cdot 2^{\f{n}{2}}b^{q}\left(\f{n}{2}\right)!q!}
= \f{3}{2bq}\cdot\f{(bq)!}{\kappa_{1}\cdot 2^{\f{bq}{2}}b^{q}\left(\f{bq}{2}\right)!q!}.
\end{align}
Let the right-hand side be
\begin{align}
\label{eq:Rbq}
R(b, q) \ceq \f{3}{2bq}\cdot\f{(bq)!}{\kappa_{1}\cdot 2^{\f{bq}{2}}b^{q}\left(\f{bq}{2}\right)!q!}.
\end{align}

The ratio of successive values of $R(b, q)$ with respect to $b$ is given by
\begin{align}
\f{R(b+1, q)}{R(b, q)} &= \f{3}{2(b+1)q}\cdot\f{((b+1)q)!}{\kappa_{1}\cdot 2^{\f{(b+1)q}{2}}(b+1)^{q}\left(\f{(b+1)q}{2}\right)!q!}\cdot
\f{2bq}{3}\cdot\f{\kappa_{1}\cdot 2^{\f{bq}{2}}b^{q}\left(\f{bq}{2}\right)!q!}{(bq)!} \\
&= 2^{-\f{q}{2}}\left(\f{b}{b+1}\right)^{q+1}\f{((b+1)q)!}{(bq)!}\cdot
\f{\left(\f{bq}{2}\right)!}{\left(\f{(b+1)q}{2}\right)!}.
\end{align}
Here,
\begin{align}
\f{((b+1)q)!}{(bq)!} = (bq+1)(bq+2)\cdots(bq+q) > (bq)^{q}. 
\end{align}
Moreover, since $q/2 \ge 1$, by \lemref{lem:gamma-at},
\begin{align}
\f{\left(\f{bq}{2}\right)!}{\left(\f{(b+1)q}{2}\right)!} &= \f{\Gamma\left(\f{bq}{2}+1\right)}{\Gamma\left(\f{bq}{2}+\f{q}{2}+1\right)} \\
&\ge \f{\Gamma\left(\f{bq}{2}+1\right)}{\left(\f{(b+1)q}{2}\right)^{\f{q}{2}}\Gamma\left(\f{bq}{2}+1\right)}
= \left(\f{2}{(b+1)q}\right)^{\f{q}{2}}.
\end{align}
Therefore,
\begin{align}
\f{R(b+1, q)}{R(b, q)} &\ge 2^{-\f{q}{2}}\left(\f{b}{b+1}\right)^{q+1}(bq)^{q}\left(\f{2}{(b+1)q}\right)^{\f{q}{2}} \\
&= \frac{b^{2q+1}q^{\f{q}{2}}}{(b+1)^{\f{3q}{2}+1}}
= \f{b}{b+1}\left(\f{b^{2}q^{\f{1}{2}}}{(b+1)^{\f{3}{2}}}\right)^{q}.
\end{align}
Since $b \ge 3$ and $q \ge 2$,
\begin{align}
\f{b^{2}q^{\f{1}{2}}}{(b+1)^{\f{3}{2}}} = \f{b^{\f{1}{2}}q^{\f{1}{2}}}{(1+\f{1}{b})^{\f{3}{2}}}
\ge \f{3^{\f{1}{2}}2^{\f{1}{2}}}{(1+\f{1}{3})^{\f{3}{2}}} = \f{9\sqrt{2}}{8} > 1,
\end{align}
therefore we obtain
\begin{align}
\f{R(b+1, q)}{R(b, q)} &\ge \f{b}{b+1}\left(\f{9\sqrt{2}}{8}\right)^{q} = \f{1}{1+\f{1}{b}}\left(\f{9\sqrt{2}}{8}\right)^{q} \\
&\ge \f{1}{1+\f{1}{3}}\left(\f{9\sqrt{2}}{8}\right)^{2} = \f{243}{128} > 1.
\end{align}
Hence $R(b, q)$ is strictly increasing in $b$ for each $q$.

The ratio of successive values of $R(b, q)$ with respect to $q$ is given by
\begin{align}
\f{R(b, q+1)}{R(b, q)} &= \f{3}{2b(q+1)}\cdot\f{(b(q+1))!}{\kappa_{1}\cdot 2^{\f{b(q+1)}{2}}b^{q+1}\left(\f{b(q+1)}{2}\right)!(q+1)!}
\f{2bq}{3}\cdot\f{\kappa_{1}\cdot 2^{\f{bq}{2}}b^{q}\left(\f{bq}{2}\right)!q!}{(bq)!} \\
&= 2^{-\f{b}{2}}\f{q}{b(q+1)^{2}}\cdot\f{\left(\f{bq}{2}\right)!}{\left(\f{b(q+1)}{2}\right)!}\cdot\f{(b(q+1))!}{(bq)!}.
\end{align}
Here,
\begin{align}
\f{(b(q+1))!}{(bq)!} = (bq+1)(bq+2)\cdots(bq+b) \ge (bq)^{b}. 
\end{align}
Moreover, since $b/2 \ge 1$, by \lemref{lem:gamma-at},
\begin{align}
\f{\left(\f{bq}{2}\right)!}{\left(\f{b(q+1)}{2}\right)!} &= \f{\Gamma\left(\f{bq}{2}+1\right)}{\Gamma\left(\f{bq}{2}+\f{b}{2}+1\right)} \\
&\ge \f{\Gamma\left(\f{bq}{2}+1\right)}{\left(\f{b(q+1)}{2}\right)^{\f{b}{2}}\Gamma\left(\f{bq}{2}+1\right)}
= \left(\f{2}{b(q+1)}\right)^{\f{b}{2}}.
\end{align}
Therefore,
\begin{align}
\f{R(b, q+1)}{R(b, q)} &\ge 2^{-\f{b}{2}}\f{q}{b(q+1)^{2}}(bq)^{b}\cdot\left(\f{2}{b(q+1)}\right)^{\f{b}{2}} \\
&= \left(\f{q}{q+1}\right)^{3}\cdot\left(\f{b^{\f{1}{2}}q}{(q+1)^{\f{1}{2}}}\right)^{b-2}.
\end{align}

When $b = 3$ and $q \ge 3$,
\begin{align}
\f{R(b, q+1)}{R(b, q)} &\ge \left(\f{q}{q+1}\right)^{3}\cdot\f{\sqrt{3}q}{(q+1)^{\f{1}{2}}}
= \left(\f{1}{1+\f{1}{q}}\right)^{3}\cdot\f{\sqrt{3}q^{\f{1}{2}}}{\left(1+\f{1}{q}\right)^{\f{1}{2}}} \\
&\ge \left(\f{1}{1+\f{1}{3}}\right)^{3}\cdot\f{\sqrt{3}\cdot 3^{\f{1}{2}}}{\left(1+\f{1}{3}\right)^{\f{1}{2}}} = \f{81\sqrt{3}}{128} > 1.
\end{align}

When $b \ge 4$, since $b-2 > 0$ and
\begin{align}
\f{b^{\f{1}{2}}q}{(q+1)^{\f{1}{2}}} &= \left(\f{bq}{1+\f{1}{q}}\right)^{\f{1}{2}} \ge \left(\f{4\cdot 2}{1+\f{1}{2}}\right)^{\f{1}{2}}
= \f{4}{\sqrt{3}} > 1,
\end{align}
we obtain
\begin{align}
\f{R(b, q+1)}{R(b, q)} &\ge \left(\f{q}{q+1}\right)^{3}\cdot\left(\f{4}{\sqrt{3}}\right)^{4-2}
= \left(\f{1}{1+\f{1}{q}}\right)^{3}\cdot\f{16}{3} \\
&\ge \left(\f{1}{1+\f{1}{2}}\right)^{3}\cdot\f{16}{3} = \f{128}{81} > 1.
\end{align}

Hence $R(b,q)$ is strictly increasing in $q$
for all $b \ge 4$, and also for $b=3$ whenever $q \ge 3$.

By direct computation using \eqref{eq:Rbq}, we obtain
\begin{align}
&R(3,2),\ R(3,3),\ R(3,4),\ R(4,2) \le 1, \\
&R(3,5),\ R(4,4),\ R(5,2) > 1.
\end{align}

Combining these values with the monotonicity of $R(b,q)$ established above,
we obtain $R(b,q)>1$, and hence $\vt{N}>\vt{D}$,
for all pairs $(b,q)$ except
\[
(b,q)=(3,2),\ (3,3),\ (3,4),\ (4,2).
\]

For these remaining pairs, 
the values of $\vt{N}$ and $\vt{D}$ are listed in
\tabref{tab:exceptionbq}.

In the case $(b,q)=(3,3)$, we still have $\vt{N} > \vt{D}$,
and therefore a dessin with trivial automorphism group exists.

For the remaining three pairs, we have $\vt{N}<\vt{D}$.
For each of them, \tabref{tab:exceptionbq} contains an explicit permutation
$y$ (with $x=(1\ 2\ \ldots\ n)$ fixed) whose corresponding dessin
has trivial automorphism group.

\begin{table}[htbp]
\centering
Assume that $x=(1\ 2\ \ldots\ n)$ is fixed. \\
\begin{tabular}{|c|c|r|r|l|c|}
\hline
$b$ & $q$ & $\vt{N}$ & $\vt{D}$ & $y$ with $\AD \cong \{ 1 \}$ & $\gen{x, y}$ \\
\hline
3 & 2 & 12 & 60 & $(1\ 2\ 4)(3\ 5\ 6)$ & $S_{5}$ \\
\hline
3 & 3 & 464 & 162 &  &  \\
\hline
3 & 4 & 38720 & 47952 & $(1\ 2\ 3)(4\ 5\ 6)(7\ 8\ 10)(9\ 11\ 12)$ & $S_{12}$ \\
\hline
4 & 2 & 276 & 384 & $(1\ 2\ 3\ 5)(4\ 6\ 7\ 8)$ & $S_{8}$ \\
\hline
\end{tabular}
\vspace{1\baselineskip}
\caption{Exceptional $(b, q)$ pairs}\label{tab:exceptionbq}
\end{table}

Therefore, every passport $[n, b^{q}, n]$ of genus at least~$2$ with $b \ge 3$ and $q \ge 2$ admits a dessin with trivial automorphism group.
\end{proof}

\begin{rem}
For $b=2$, the above argument does not apply.

Indeed, when $b=2$ and the genus is at least~$2$, we have $n = 2q$ with $q \ge 4$.
Since $2 \mid n$, by \propref{prop:Dk} and \propref{prop:D}, we obtain
\begin{align}
D &= \bigcup_{\ell\colon \text{prime},\, \ell\, \mid\, n} D_{\f{n}{\ell}}, \\
\vt{D} &\ge \vt{D_{\f{n}{2}}} = \vt{D_{q}} = 2^{q}q!.
\end{align}
Then, by \propref{prop:tbq},
\begin{align}
\label{eq:b2dt}
\f{\vt{D}}{\vt{T}} &\ge 2^{q}q!\left(\f{n!}{2^{q}q!}\right)^{-1} = \f{2^{2q}(q!)^{2}}{(2q)!}.
\end{align}

Moreover,
\begin{align}
2^{2q} = (1+1)^{2q} = \sum_{k=0}^{2q}\binom{2q}{k}.
\end{align}
Since all terms in the sum are positive, it follows that
\begin{align}
2^{2q} > \binom{2q}{q} = \f{(2q)!}{(q!)^{2}},
\end{align}
hence
\begin{align}
\f{2^{2q}(q!)^{2}}{(2q)!} > 1.
\end{align}
Combining this with \eqref{eq:b2dt}, we obtain
\begin{align}
\f{\vt{D}}{\vt{T}} & > 1.
\end{align}
Thus, by the definition \eqref{eq:tncd}, we have $\vt{N} \le \vt{T} < \vt{D}$.
Therefore, the above method cannot establish $\vt{N} > \vt{D}$.
This indicates that a different approach is required for the case $b=2$.
\end{rem}

\subsection{Result}

Thus we obtain an alternative proof of \cite[Theorem~7.6]{Ohnishi26}.

\begin{thm}\cite[Theorem~7.6]{Ohnishi26}
\label{thm:class2}
If a uniform passport $[n, b^{q}, n]$ $(q \ge 1)$ has genus at least~$2$, then it admits a \dde with a trivial automorphism group.

The same statement holds for the passports $[b^{q}, n, n]$ and $[n, n, b^{q}]$.
\end{thm}

\begin{proof}
Since the automorphism group is symmetric with respect to black vertices, white vertices, and faces,
it suffices to prove the statement for the passport $[n, b^{q}, n]$ ($n = bq$).

By \eqref{eq:genus}, the genus is

\begin{align}
\f{n - (1 + q + 1)}{2} + 1 = \f{n - q}{2} = \f{q(b-1)}{2} \ge 2.
\end{align}
Therefore, we have $b \ge 2$.

The case $b = n$ ($q = 1$) was proved in \cite[Proposition~7.1]{Ohnishi26}.

The case $b = 2$ was proved in \propref{prop:c2b2}.

The case $3 \le b < n \ (q \ge 2)$ was proved in \propref{prop:c2bge3}.
\end{proof}

\FloatBarrier
\subsection*{Acknowledgements}

I would like to express my sincere gratitude to Associate Professor Yasuhiro Wakabayashi for his detailed guidance
and insightful advice throughout my research.

I am also deeply grateful to all the members of the Wakabayashi Laboratory for their continued support and
valuable discussions.

\vspace{-.5\baselineskip}

\bibliographystyle{amsalpha}
\bibliography{References}

\end{document}